\documentclass{amsart}
\usepackage{amsmath}
\usepackage{amsthm}
\usepackage{amsfonts}
\usepackage{amssymb}
\newtheorem{Theorem}{Theorem}[section]
\newtheorem{Proposition}[Theorem]{Proposition}
\newtheorem{Lemma}[Theorem]{Lemma}
\newtheorem{Corollary}[Theorem]{Corollary}
\theoremstyle{definition}
\newtheorem{Definition}{Definition}[section]
\theoremstyle{remark}

\numberwithin{equation}{section}

\newcommand{\Z}{{\mathbb Z}}
\newcommand{\R}{{\mathbb R}}
\newcommand{\C}{{\mathbb C}}
\newcommand{\N}{{\mathbb N}}

\newcommand{\bbR}{{\mathbb{R}}}

\newcommand{\bbP}{{\mathbb{P}}}

\newcommand{\bbC}{{\mathbb{C}}}

\begin{document}

\title[Absolutely continuous spectrum]{The absolutely continuous spectrum
of Jacobi matrices}

\author{Christian Remling}

\address{Mathematics Department\\
University of Oklahoma\\
Norman, OK 73019}

\email{cremling@math.ou.edu}

\urladdr{www.math.ou.edu/$\sim$cremling}

\date{June 7, 2007; revised August 11, 2010}



\thanks{CR's work has been supported by NSF grant DMS 0758594}
\begin{abstract}
I explore some consequences of a groundbreaking result of Brei\-mes\-ser and Pearson
on the absolutely continuous spectrum of one-dimensional Schr\"o\-din\-ger operators.
These include an Oracle Theorem that predicts the potential and
rather general results
on the approach to certain limit potentials. In particular, we prove a Denisov-Rakhmanov
type theorem for the general finite gap case.

The main theme is the following: It is extremely difficult to produce
absolutely continuous spectrum in one space dimension and thus its existence
has strong implications.
\end{abstract}

\maketitle

\section{Introduction and statement of main results}
\subsection{Introduction}
This paper deals with one-dimensional discrete Schr\"odinger operators on $\ell_2$,
\begin{equation}
\label{so}
(Hu)(n)=u(n+1)+u(n-1)+V(n)u(n),
\end{equation}
with some absolutely continuous spectrum. We will also consider Jacobi matrices,
\[
(Ju)(n) = a(n)u(n+1) + a(n-1)u(n-1) + b(n)u(n) ;
\]
these of course include \eqref{so} as the special case $a(n)=1$.

The purpose of this paper is to explore a stunning result of Breimesser and Pearson \cite{BP1,BP2} (which
seems to have gone almost unnoticed). I will give a reformulation in Theorem \ref{TBP} below,
which (I believe) should help to clarify the significance of the brilliant work of Breimesser-Pearson.
In fact, it seems
to me that \cite{BP1,BP2} reveal new fundamental properties of the absolutely continuous spectrum.
The situation is perhaps reminiscent of the
reevaluation of the singular continuous spectrum some ten years ago
(shown to be ubiquitous, contrary to then common popular belief)
\cite{dRMS,Gor,HKnSim,JitSim,Psc,Simsc}.

As is very well known, the absolutely continuous spectrum is that part of the spectrum that has
the best stability properties under small perturbations. Once this is admitted, it turns
out that it is extremely difficult to produce absolutely continuous spectrum in one space dimension
in any other way (other than a small perturbation of one of the few known examples). This is the
main message of this paper.
(I used to believe the exact opposite: absolutely continuous spectrum is what you normally get unless
something special happens, but this now turns out to be a gross misinterpretation.)

In addition to the work of Breimesser and Pearson, a second important source of inspiration
for this paper is provided by Kotani's theory of the absolutely continuous spectrum of ergodic
systems of Schr\"odinger operators \cite{Kotac,Kot,Kotfin,Kot97}. In fact, much of what we will do here
may be viewed as a Kotani like theory, but for individual, non-ergodic operators.
\subsection{Comment on notation}
We will discuss these issues in more detail in a moment, but let me first introduce a notational
convention that will be used throughout this paper.
Everything we do here
works in the general (Jacobi) setting, but the need to deal with two sequences of coefficients
$a(n)$, $b(n)$ often makes the notation awkward. So it might seem wise to only deal with the
Schr\"odinger case, but this is not an ideal solution either
because sometimes the greater generality of the
Jacobi setting is essential. I have decided on a perhaps somewhat unusual remedy against
this predicament: Since usually the extension to the Jacobi case is obvious, I will simply
work in the Schr\"odinger operator setting most of the time.
Occasionally, I will have to switch to the Jacobi
case, though. For example,
Sections 5, 6, 7 don't make much sense without this generality, and in Sect.~3, it's
not totally clear how to incorporate the $a(n)$'s. However, I will usually
quickly switch back to the Schr\"odinger notation when feasible. I hope that this
leads to a more easily readable presentation without confusing the reader too much.

If necessary, I will also
use all previous results as if they had been proved for Jacobi operators.
In other words, everything in this paper is (at least implicitly) asserted for the general Jacobi case.

So there are two extreme ways of reading this paper:\\
(1) \textbf{Schr\"odinger reader:} specialize to $a(n)=1$ and identify $b(n)=V(n)$ whenever you see
coefficients $a(n)$, $b(n)$;\\
(2) \textbf{Jacobi reader:} replace $V(n)$ with $(a(n),b(n))$ throughout and make other
adjustments as necessary (frequently, no such additional adjustments are necessary).
Somewhat more detailed instructions for the Jacobi reader will be given as we go.
\subsection{The Oracle Theorem}
The basic result of this paper is Theorem \ref{TBP} below, but let me begin
the discussion by mentioning two
consequences that are particularly accessible:
\begin{Theorem}
\label{T1.1}
Suppose that the (half line) potential $V(n)$ takes only
finitely many values and $\sigma_{ac}\not=\emptyset$.
Then $V$ is eventually periodic: There exist $n_0, p\in\N$ so that
\[
V(n+p)=V(n) \quad \textrm{for all } n\ge n_0 .
\]
\end{Theorem}
For ergodic potentials, this is a well known Theorem of Kotani \cite{Kotfin}. That it holds for arbitrary
operators came as a mild surprise, at least to me.
(Recall also that by our general convention, the same statement
holds for Jacobi operators: if $\sigma_{ac}(J)\not=\emptyset$,
then eventually $a(n+p)=a(n)$, $b(n+p)=b(n)$
for some $p\in\N$.)

Theorem \ref{T1.1} is a consequence of the following more general result,
which says that there are universal oracles that
will predict future values of potentials with some absolutely continuous spectrum with any desired
accuracy, based on (partial) information on past values.
\begin{Theorem}[The Oracle Theorem]
\label{Tor}
Let $C>0$, $\epsilon>0$, and let $A\subset\R$ be a Borel set of positive Lebesgue measure.
Then there exist $L\in\N$ and a smooth function
\[
\Delta: [-C,C]^{L+1} \to [-C, C]
\]
(the {\em oracle}), such that the following holds: For any (half line) potential $V$ with
$\|V\|_{\infty} \le C$ and $\Sigma_{ac}(V) \supset A$, there exists $n_0\in \N$ so that
for all $n\ge n_0$,
\[
\left| V(n+1) - \Delta\left( V(n-L), V(n-L+1), \ldots , V(n)\right) \right|
< \epsilon .
\]
\end{Theorem}
Here, we use the symbol $\Sigma_{ac}$ to denote an essential support of the absolutely continuous
part of the spectral measure. In other words, the measures $d\rho_{ac}$ and $\chi_{\Sigma_{ac}}\, dt$ have
the same null sets. This condition determines $\Sigma_{ac}$ up to sets of (Lebesgue) measure zero.
The absolutely continuous spectrum, $\sigma_{ac}$, is the essential closure of $\Sigma_{ac}$.

\textit{Jacobi reader: }Interpret the assumption that $\|V\|_{\infty}\le C$ as $V\in\mathcal V_+^C$;
this will be explained in more detail below. The oracle will now predict $(a(n+1),b(n+1))$, as a function
of $(a(j),b(j))$ for $n-L\le j\le n$.

Note that only $n_0$ depends on $V$ (this, of course, is inevitable, because we can always modify
$V$ on a finite set without affecting the absolutely continuous spectrum); the oracle $\Delta$ itself
is universal and works for any potential $V$ satisfying the assumptions.

Can we also predict $V$ if we just know that $V$ has \textit{some} absolutely continuous spectrum?
Clearly, the answer to this is no because any periodic $V$
has non-empty absolutely continuous spectrum, and it is certainly not possible to make any predictions
about the next value of an arbitrary periodic potential, based on a finite number of previous values
(the period could simply be larger than that number). Therefore, the oracle $\Delta$ must depend on
the set $A$, which serves as a lower bound for $\Sigma_{ac}$.

Theorem \ref{T5.2} below will further clarify this issue. It will show how exactly things can go
wrong if we don't have some a priori information on $\Sigma_{ac}$.

Theorem \ref{Tor} will be proved in Sect.\ 4.
Again, we can view the Oracle Theorem as a general version of a famous
result of Kotani \cite{Kotac,Kot} on ergodic operators
(\textit{ergodic }potentials with some absolutely continuous spectrum
are deterministic).

We can confirm right away that
Theorem \ref{T1.1} indeed is an immediate consequence
of the Oracle Theorem.
\begin{proof}[Proof of Theorem \ref{T1.1}]
By choosing $\epsilon>0$ small enough, we can use an oracle to (eventually) predict $V(n)$ exactly,
given the previous $L+1$ values of $V$. But there are only finitely many different blocks
of size $L+1$, so after a while, things must start repeating themselves.
\end{proof}

Please see also Corollary \ref{C1.3} below for another illustration of the Oracle Theorem in action.
\subsection{The basic result}
Let me now present the basic theorem of this paper: the reformulation of
Theorem 1 from \cite{BP1}. This result, in its original version (but for
discrete rather than continuous operators), will be formulated as Theorem \ref{TBPorig}
below; the proof will be given in Appendix A.

We consider the space $\mathcal V^C$ of bounded (whole line) potentials $|V(n)|\le C$.
This becomes a compact
topological space if endowed with the product topology. In fact, the space is metrizable;
one possible choice for the metric is
\[
d(V,W) = \sum_{n=-\infty}^{\infty} 2^{-|n|} |V(n)-W(n)| .
\]
More generally, we will frequently have occasion to consider
half line and whole line potentials simultaneously,
and thus we extend the definition of $d$ as follows: If $V:A\to [-C,C]$, $W:B\to [-C,C]$,
where $A,B\subset\Z$, then we simply put
\[
d(V,W) = \sum_{n\in A\cap B} 2^{-|n|} |V(n)-W(n)| .
\]
The typical case is: one set equals $\Z$, the other is a half line. We will
also use the modified notation $\mathcal V_{\pm}^C$ to refer to half line potentials,
defined on $\Z_+$ and $\Z_-$, respectively, where
\begin{align*}
\Z_+ & = \{ 1, 2, 3, \ldots \}, \\
\Z_- & = \{ \ldots, -2, -1, 0 \} .
\end{align*}

Recall also that by the Simon-Spencer Theorem \cite{SimSp} (see also \cite{Dom}
and \cite[Theorem 4.1]{LS})
a (half line) potential
$V$ is automatically bounded if $\sigma_{ac}\not=\emptyset$.

\textit{Note to the Jacobi reader: }For $C>0$,
define $\mathcal V^C$ as the space of sequences $(a(n),b(n))$
satisfying $(C+1)^{-1}\le a(n) \le C+1$, $|b(n)|\le C$.
(Again, if we assume that $a(n) \le C_1$ and if $\sigma_{ac}\not=\emptyset$,
the other inequalities follow automatically by the Simon-Spencer argument.)
In the definition of $d$, replace
$|V(n)-W(n)|$ by $|a(n)-a'(n)|+|b(n)-b'(n)|$ (say). We will frequently refer to a $V\in\mathcal V^C$ as a
\textit{bounded potential; }the Jacobi reader will have to interpret this term as explained above.
By the same token, we will often use the term \textit{potential} for what in the Jacobi case
would really be a sequence of coefficients $(a(n),b(n))$.

The absolutely continuous spectrum as well as the essential spectrum are independent of the
behavior of the potential on any finite set, so it seems natural to study the $\omega$ limit
set of a given bounded (half line) potential $V$ under the shift map $S$
when one is interested in these parts of the
spectrum. Thus we define
\[
\omega(V) = \{ W \in \mathcal V^C : \textrm{There exists a sequence }n_j\to\infty
\textrm{ such that } d(S^{n_j}V, W) \to 0 \} ;
\]
as already explained, $S$ denotes the shift map, that is,
\[
(S^kV)(n)=V(n+k) .
\]
Note that here indeed $V$ is a half line potential, while the limits $W$ are whole line potentials.
These $\omega$ limit sets will play a very important role in this paper; they have also been
studied by Last and Simon in \cite{LS,LS2} (where they are called right limits).

We record some well known basic properties.
\begin{Proposition}
\label{P1.1}
$\omega(V)\subset\mathcal V^C$ is compact, non-empty, and $S$ is a homeomorphism on $\omega(V)$.
Moreover,
\[
d(S^n V, \omega(V) ) \to 0 \quad\quad (n\to\infty) .
\]
\end{Proposition}
The easy proof of Proposition \ref{P1.1} will be given in Sect.~3.

The key to everything is the following definition: Let $W$ be a bounded whole line potential.
Write $m_{\pm}(z)$
for the Titchmarsh-Weyl $m$ functions of the operator restricted to the half lines $\Z_+$ and
$\Z_-$, respectively. (Precise formulae for $m_{\pm}$ will be given in Sect.~3 below.)
\begin{Definition}
\label{D1.1}
Let $A\subset\R$ be a Borel set. Then we call $W\in\mathcal V^C$ \textit{reflectionless} on $A$ if
\begin{equation}
\label{1.5}
m_+(t) = - \overline{m_-(t)} \quad\quad \textrm{for (Lebesgue) almost every }t\in A .
\end{equation}
We will also use the notation
\[
\mathcal R (A) = \{ W\in\bigcup_{C>0}\mathcal V^C : W \textrm{ reflectionless on }A \} .
\]
\end{Definition}

Of course, this requirement is non-vacuous only if $A$ has positive Lebesgue measure.
The condition \eqref{1.5} can be reformulated in a number of ways.
I gave it in the form most suitable for the purposes of this paper,
but for a perhaps more immediately accessible
definition, I will also mention that \eqref{1.5} is equivalent to
$\textrm{Re }G(n,t)=0$ for almost every $t\in A$ and all $n\in\Z$, where
\[
G(n,z) = \langle \delta_n , (J-z)^{-1} \delta_n \rangle
\]
is the Green function of the whole line Jacobi matrix with coefficients $W$.
See also \cite[Lemma 8.1]{Teschl} for further information on \eqref{1.5}.

\textit{Warning: }Some authors use a more restrictive definition and
call a potential reflectionless if (in our terminology) it is reflectionless
on $\sigma_{ess}$. For the purposes of this paper, it is essential to work with Definition \ref{D1.1}.

We are now finally ready to state our reformulation of the Breimesser-Pearson Theorem.
\begin{Theorem}
\label{TBP}
Let $V$ be a bounded (half line) potential, and, as above, let $\Sigma_{ac}$ be the essential support
of the absolutely continuous part of the spectral measure. Then
\[
\omega(V) \subset \mathcal R (\Sigma_{ac}) .
\]
\end{Theorem}
Since not many potentials are reflectionless, this gives very strong restrictions
on the structure of potentials with some absolutely continuous spectrum, one of these
being the Oracle Theorem. Other applications of Theorem \ref{TBP}
will be discussed in a moment. For the full picture, please also see Sect.~7, which has
two examples that illustrate what is \textit{not }true in this context.
\subsection{Further consequences of Theorem \ref{TBP}}
Note how ridiculously easy it is to prevent absolutely continuous spectrum.
For example, if $V$ contains arbitrarily large chunks $(W(-R),\ldots , W(R))$ of a potential $W$ that
is not reflectionless on any set of positive measure (and the potentials $W$
that do \textit{not} have this property form a very small subclass), then $\sigma_{ac}(V)=\emptyset$.
Elaborating further on this simple remark, we obtain the following
result that again shows how easy it is to destroy absolutely continuous spectrum.
\begin{Corollary}
\label{C1.3}
Let $U$ be a perturbation that has the following property:
There exists a subsequence $n_j\to\infty$ so that
\[
\limsup_{j\to\infty} \left| U(n_j) \right| > 0 ,
\]
but
\[
\lim_{j\to\infty} |U(n_j-k)| = 0
\]
for all $k\ge 1$.
Then $\Sigma_{ac}(V+U) \cap \Sigma_{ac}(V) = \emptyset$ for any (half line) potential $V$.

In particular, this conclusion holds for every perturbation $U$ of the form
\[
U(n) = \sum_{j=1}^{\infty} u_j \delta_{n,n_j},
\quad\quad n_j-n_{j-1}\to\infty, \quad\limsup_{j\to\infty} |u_j| > 0 .
\]
\end{Corollary}
More precisely, the claim is that one can choose representatives with empty intersection
(recall that $\Sigma_{ac}$ is only determined up to sets of measure zero).
To obtain the Jacobi version, interpret $U(n)=(\alpha(n),\beta(n))$ and $|U(n)|=|\alpha(n)|
+|\beta(n)|$; we can allow negative and/or unbounded
coefficients here, but then we must require that both the original and
the perturbed operator be Jacobi matrices with bounded $a$'s,
that is, if $V=(a(n),b(n))$, we demand that $0<a(n)\le C$ and $0<a(n)+\alpha(n)\le C$
for some $C>0$.

Of course, Corollary \ref{C1.3} is a result very much in the spirit of Pearson's classic \cite{Psc}
(sparse perturbations destroy absolutely continuous spectrum), but it is much
more general. See also \cite[Sect.~5]{BP1} for the $V=0$ case.

To get a feeling for the power of Theorem \ref{TBP},
it is also instructive to compare the cheap proof below with the
traditional approach to analyzing sparse potentials
(with $V=0$),
which uses a considerable amount of (so-called)
hard analysis. See, for example, \cite{KLS,KruRem,Mol,Psc,Remsparse,Remsc,Zla}.
\begin{proof}[Proof of Corollary \ref{C1.3}]
If either $V\notin\bigcup_{C>0} \mathcal V^C$ or $V+U\notin\bigcup_{C>0}\mathcal V^C$,
then the corresponding operator
has empty absolutely continuous spectrum by the Simon-Spencer Theorem.
Otherwise, $V,V+U\in\mathcal V^C$ for some $C>0$, and then
$A=\Sigma_{ac}(V) \cap \Sigma_{ac}(V+U)$ cannot have positive measure,
because then the Oracle Theorem would provide oracles that work for both $V$ and $V+U$.
This is impossible because $|U(n_j)|\ge\epsilon>0$ on a
suitable subsequence, but $U$ is small on long intervals
to the left of these points, so no (continuous) oracle with
sufficiently high accuracy can
predict both $V$ and $V+U$ correctly.
\end{proof}
It is not necessary to use the Oracle Theorem here. Corollary \ref{C1.3} also follows directly from
Theorem \ref{TBP} if we make use of a standard uniqueness property of reflectionless potentials,
which is given in Proposition \ref{P4.1}(c) below. When we prove the Oracle Theorem in Sect.~4, we will
see that this is essentially a rewording of the original argument.

As another immediate consequence of Theorem \ref{TBP}, we effortlessly recover an important result
of Last and Simon on the semicontinuity of $\Sigma_{ac}$:
\begin{Corollary}[Last-Simon \cite{LS}]
\label{C1.1}
If $W\in\omega(V)$, then $\Sigma_{ac}(W_{\pm})\supset\Sigma_{ac}(V)$.
\end{Corollary}
Here,
\[
W_{\pm}:= W \big|_{\Z_{\pm}}
\]
denote the half line restrictions of $W$, and by $\Sigma_{ac}(W_{\pm})$, we mean
the essential supports of the absolutely continuous parts of the spectral measures of
the corresponding half line operators.
\begin{proof}
By Theorem \ref{TBP}, $W\in\mathcal R(\Sigma_{ac}(V))$. By Proposition \ref{P4.1}(b) below,
this implies the claim.
\end{proof}
We now move on to a new topic related to Theorem \ref{TBP}.
First of all, recall that the essential spectrum satisfies the opposite inclusion.
In fact, a routine argument
using Weyl sequences even shows that if $W\in\omega(V)$, then
\begin{equation}
\label{1.1}
\sigma(W_{\Z}) \subset \sigma_{ess}(V) .
\end{equation}
We have written $W_{\Z}$ to emphasize the fact that we need to consider the \textit{whole line} operator
associated with $W$ here (but of course $V$ continues to be a half line potential); \eqref{1.1} is
certainly not correct for the half line operators generated by $W$. See also \cite{GeoI,LS2} for
more sophisticated results on the relation between $\omega(V)$ and $\sigma_{ess}(V)$.

We now realize that we obtain especially strong restrictions on the possible
$\omega$ limit sets $\omega(V)$ if $\Sigma_{ac}(V)=\sigma_{ess}(V)$, and this
set is essentially closed (that is, its intersection with an arbitrary open set
is either empty or of positive Lebesgue measure).
Then any $W\in\omega(V)$ must satisfy
\begin{equation}
\label{1.6}
\Sigma_{ac}(W_+)=\Sigma_{ac}(W_-)=\sigma(W_{\Z})=\sigma_{ess}(V) .
\end{equation}
Again, it would be more careful to say that it is possible to choose
representatives of $\Sigma_{ac}(W_{\pm})$ satisfying \eqref{1.6}. It is also helpful to
recall here that there is a decomposition method for both $\sigma_{ess}$ and $\sigma_{ac}$,
that is, if $U$ is a whole line potential and $U_{\pm}$ denote, as usual, the half line
restrictions, then
\[
\sigma_{ess}(U)=\sigma_{ess}(U_-)\cup\sigma_{ess}(U_+) ,
\]
and similarly for $\sigma_{ac}$. This is in fact obvious because cutting into two
half lines (at $n=0$) amounts to replacing $a(0)$ by $0$, which is a rank two perturbation.

In this generality, the theme of studying \eqref{1.6} was introduced and investigated by Damanik, Killip,
and Simon in \cite{DKS};
see especially \cite[Theorem 1.2]{DKS}. By using Theorem \ref{TBP},
we can go beyond the results of \cite{DKS}.

Indeed, Theorem \ref{TBP} gives the strong additional condition that
$W\in\mathcal R(\Sigma_{ac})$.
For a quick illustration of how this can be used,
recall that the only $W=(a_0,b_0)\in\mathcal R([-2,2])$ with $\sigma(W_{\Z})=[-2,2]$ is
the free Jacobi matrix
\begin{equation}
\label{1.8}
a_0(n)=1, \quad b_0(n)=0 .
\end{equation}
This is well known (see, for instance, \cite[Corollary 8.6]{Teschl}), but we will also
provide a proof here at the end of Sect.~6.

We can now use this to easily recover another important result, which
is due to Denisov; earlier work in this direction was done by Rakhmanov \cite{Rakh}.
\begin{Corollary}[Denisov \cite{Den}]
\label{C1.2}
If $J$ is a bounded (half line) Jacobi matrix satisfying
\[
\sigma_{ess}(J)=\Sigma_{ac}(J)=[-2,2],
\]
then $a(n)\to 1$, $b(n)\to 0$ as $n\to\infty$.
\end{Corollary}
\begin{proof}
By Proposition \ref{P1.1}, $d(S^n(a,b),\omega(J))\to 0$ as $n\to\infty$. By Theorem \ref{TBP}
and the observations made above (see \eqref{1.6}),
any $W=(a_0,b_0)\in\omega(J)$ must satisfy
$W\in\mathcal R([-2,2])$ and $\sigma(W)=[-2,2]$. But as just pointed out,
the only such $W=(a_0,b_0)$ is given by \eqref{1.8}.
\end{proof}
Clearly, this idea can be pushed further. We automatically
obtain valuable information on the asymptotics of $V=(a,b)$ provided we can
extract sufficiently detailed information about the possible elements
of $\omega(V)$ from \eqref{1.6} and the statement of Theorem \ref{TBP}.
In particular, this approach works very smoothly in the general finite gap case
(this is the usual terminology, meaning finitely many gaps).
\begin{Theorem}
\label{TDR}
Suppose that $J$ is a bounded (half line) Jacobi matrix satisfying
\[
\sigma_{ess}(J)=\Sigma_{ac}(J)= \bigcup_{j=0}^N
[\alpha_j,\beta_{j+1}]=:E \quad\quad (\alpha_j<\beta_{j+1}<\alpha_{j+1}) .
\]
Then $d(S^n(a,b), \mathcal T_N ) \to 0$, where $\mathcal T_N=\mathcal T_N(E)$
denotes the set of finite gap Jacobi coefficients with spectrum $E$.
\end{Theorem}
The statement of Theorem \ref{TDR} may sound a bit vague, but
the set $\mathcal T_N$ actually has a very explicit description, and a considerable amount
of work has
been done on these operators (and their continuous analogs). See, for example,
\cite{BGHT,Dub,KrRem,McK,Mum,Teschl,vMoer,vMoerMum}; also notice that these finite gap operators
have been popular in several different (but overlapping) areas, including spectral theory,
integrable systems, and algebraic geometry.

It is not hard to prove that with the topology induced by
$d$, the set $\mathcal T_N\subset\mathcal V^C$ is homeomorphic to an $N$-dimensional torus (thus the
notation). Here, we now simply define
$\mathcal T_N(E)$ as the set of (whole line) potentials $W=(a,b)\in\mathcal R(E)$
that satisfy $\sigma(W)=E$. This direct proof does not require any of the machinery just mentioned; we will
give it in Sect.~6.

In \cite{DKS}, Damanik, Killip, and Simon prove Theorem \ref{TDR} in the special
case where the coefficients  $(a,b)\in\mathcal T_N(E)$ are periodic. This imposes restrictions
on the spectrum $E$ and is not the generic case; in general,
the coefficients are only quasi-periodic.
So Theorem \ref{TDR} generalizes \cite[Theorem 1.2]{DKS}; moreover, the approach
via Theorem \ref{TBP} gives a rather elegant proof and seems suitable for further
generalization. We can in fact immediately go beyond Theorem \ref{TDR} if we make use of
work of Sodin and Yuditskii \cite{SY}. However, it seems that for a full understanding of
these phenomena, more detailed knowledge about the reflectionless potentials involved
here is needed.
\subsection{Organization of this paper}
The proofs of Theorem \ref{TBP} and Proposition \ref{P1.1} will be given in Sect.\ 3.
In Sect.~2, we discuss some preparatory material. The Oracle Theorem is proved in Sect.~4.
Small limit sets $\omega(V)$ are the subject of Sect.~6; in particular, Theorem \ref{TDR} is
proved there. Before we can do that, we present additional preparatory material on reflectionless
potentials in general in Sect.~5. This also gives us the opportunity to present two counterexamples
(see Theorems \ref{T5.1}, \ref{T5.2})
that address issues that were briefly mentioned above in our discussion of the Oracle Theorem.

Other examples that address the question of whether stronger statements (than
Theorem \ref{TBP}) might be possible are presented in Sect.~7. While this material is not,
strictly speaking, needed for the development of the results discussed above, it is
certainly advisable to take a quick look at this section at an early stage.

The appendix has a complete proof of the original result of
Breimesser and Pearson, which is formulated as Theorem \ref{TBPorig} below.
It seems appropriate to include a full proof here because of the central importance
of this result. Also, Breimesser and Pearson work in the continuous
setting and it thus seems useful to have a complete proof for the
discrete case written up, too. Of course, I should also be quick to emphasize that while
I do make a few minor changes in the details of the presentation, the overall strategy is exactly
the same as in the original treatment of Breimesser-Pearson \cite{BP1,BP2}.

I will also try to achieve a fuller understanding of the Breimesser-Pearson Theorem by
trying to pinpoint the decisive facts that make the proof work. As a result of this,
I have now come to the conclusion that the following
identity seems to be at the heart of the matter:
\begin{equation}
\label{1.11}
\begin{pmatrix} 1 & 0 \\ 0 & -1 \end{pmatrix} \begin{pmatrix} z & 1 \\ -1 & 0 \end{pmatrix}
\begin{pmatrix} 1 & 0 \\ 0 & -1 \end{pmatrix} = \begin{pmatrix} z & -1 \\ 1 & 0 \end{pmatrix} .
\end{equation}

Perhaps one final point deserves mention here: In the appendix, I will take the
opportunity to advertise the matrix formalism for handling the linear fractional transformations
(viewed projectively)
that are so central in Weyl theory. This notation, of course, is in common use in other areas and actually
has been used in this context, too \cite{FHS,Simopuc2}, but most
presentations of Weyl theory still give uninspired computational verifications of facts that
become much clearer once the matrix formalism is adopted. See also \cite{RemWeyl} for
further information.

\textit{Acknowledgment: }Special thanks are due to Barry Simon for helpful comments on
this paper and a lot more.
\section{Convergence of Herglotz functions}
The goal of this section is to relate convergence in value distribution, as defined
in Definition \ref{D2.1} below, to other more familiar notions of convergence. This
will become important later because the Breimesser-Pearson Theorem is formulated in these terms.

We denote the set of Herglotz functions by $\mathcal H$, that is,
\[
\mathcal H = \{ F: \bbC^+ \to \bbC^+ : F\textrm{ holomorphic} \} ;
\]
here, $\C^+=\{z\in\C : \textrm{Im }z>0\}$ is the upper half plane in $\C$.

The Herglotz representation theorem says that $F\in \mathcal H$ precisely if $F$
is of the form
\begin{equation}
\label{Hrep}
F(z) = a + \int_{\bbR_{\infty}} \frac{1+tz}{t-z}\, d\nu(t) ,
\end{equation}
with $a\in\bbR$, and $\nu\not= 0$ is a finite, positive Borel measure on $\bbR_{\infty}
=\bbR \cup \{ \infty \}$. Here, we equip $\R_{\infty}$ with the topology of the $1$-point
compactification of $\R$. Both $a$ and $\nu$ are uniquely determined by $F\in\mathcal H$.

If we let
\[
b=\nu(\{ \infty \}),\quad\quad d\rho(t) = (1+t^2)\chi_{\bbR}(t)\, d\nu(t) ,
\]
then \eqref{Hrep} takes the more familiar form
\[
F(z) = a + bz + \int_{-\infty}^{\infty} \left( \frac{1}{t-z} - \frac{t}{t^2+1} \right) \, d\rho(t) .
\]
However, since it's nice to have finite measures on a compact space, \eqref{Hrep} is often
more convenient to work with.

The two most frequently used notions
of convergence in this context are uniform convergence of the Herglotz functions
on compact subsets of $\bbC^+$ and weak $*$ convergence of the measures $\nu$. It is well known
that these are essentially equivalent; see Theorem \ref{T2.1} below.

The work of Pearson, partly in collaboration with Breimesser \cite{BP1,BP2,P,Pearvd},
suggests to also introduce \textit{convergence in value distribution, }as follows:
\begin{Definition}
\label{D2.1}
If $F_n, F\in \mathcal H$, we say that $F_n\to F$ in value distribution if
\begin{equation}
\label{convvd}
\lim_{n\to\infty} \int_A \omega_{F_n(t)}(S)\, dt = \int_A \omega_{F(t)}(S)\, dt
\end{equation}
for all Borel sets $A, S \subset \bbR$, $|A|<\infty$.
\end{Definition}

Here, for $z=x+iy\in\bbC^+$,
\begin{equation}
\label{harmmeas}
\omega_z(S) = \frac{1}{\pi} \int_S \frac{y}{(t-x)^2 + y^2}\, dt
\end{equation}
denotes harmonic measure in the upper half plane, and if $G\in\mathcal H$, $t\in\bbR$, we define
$\omega_{G(t)}(S)$ as the limit
\[
\omega_{G(t)}(S) = \lim_{y\to 0+} \omega_{G(t+iy)}(S) .
\]
Since $z\mapsto \omega_{G(z)}(S)$ is a non-negative harmonic function on $\C^+$,
the limit exists for almost every $t\in\R$.
In particular, the integrands from \eqref{convvd} are now defined almost everywhere.

It is also helpful to recall the following:
If $G\in\mathcal H$, then $G(t)\equiv \lim_{y\to 0+} G(t+iy)$ exists for almost every $t\in\R$.
If, moreover, $\textrm{Im }G(t)>0$, then the dominated convergence theorem
implies that $\lim_{y\to 0+} \omega_{G(t+iy)}(S)$ exists for arbitrary $S$
and coincides with the direct definition where we just substitute $G(t)$ for $z$ in \eqref{harmmeas}.

On the other hand, if $G(t)$ (exists and) is real, then
\[
\lim_{y\to 0+} \omega_{G(t+iy)}(S) = \begin{cases} 0 & G(t) \notin \overline{S} \\
1 & G(t) \in \mathring{S} \end{cases} .
\]
So for nice sets $S$ (and away from the boundary), $\omega_{G(t)}(S)$ is essentially
$\chi_S(G(t))$ if $G(t)\in\R$. This observation also explains the terminology: If $G(t)\in\R$
for almost every $t\in A$, then
\begin{equation}
\label{2.8}
\int_A \omega_{G(t)}(S)\, dt = \left| \{ t\in A: G(t)\in S \} \right|
\end{equation}
gives information on the distribution of the (boundary) values of $G$.
To completely prove \eqref{2.8},
first note that this formula holds for intervals $S$ by the discussion above
and the fact that $| G^{-1}(\{ a\} ) |=0$ for all $a\in\R$. Now both sides of \eqref{2.8} define
measures on the Borel sets $S\subset\R$:
this is obvious for the right-hand side, and as for the left-hand side,
the most convenient argument is to just refer to formula \eqref{spa} below.
Thus we obtain \eqref{2.8} for arbitrary Borel sets $S$.

One final word of caution is in order: It is not necessarily true that $\omega_{G(t)}(S)$ only depends on
the set $S$ and the number $G(t)$, as the notation might suggest. However, the
above remarks show that this statement is almost true, and no difficulties will be caused by
this rather subtle point.

Notice that a limit in value distribution, if it exists, is unique: If $F_n\to F$ but also
$F_n \to G$ in value distribution, then, by Lebesgue's differentiation theorem,
$\omega_{F(t)}(S) = \omega_{G(t)}(S)$ for almost every $t\in\R$ for fixed $S$.
However, a countable collection of sets $S_n$ clearly suffices to recover $F(t)$ from
the values of $\omega_{F(t)}(S_n)$ (we can use the open intervals with rational endpoints, say),
so it in fact follows that $F(t)=G(t)$ for almost every $t\in\R$. But this is what we claimed
because Herglotz functions are uniquely determined by their boundary values on a set of positive measure.

\begin{Theorem}
\label{T2.1}
Suppose that $F_n, F\in\mathcal H$, and let $a_n$, $a$, and $\nu_n, \nu$ be the associated
numbers and measures, respectively,
from representation \eqref{Hrep}. Then the following are equivalent:\\
(a) $F_n(z)\to F(z)$ uniformly on compact subsets of $\bbC^+$;\\
(b) $a_n\to a$ and $\nu_n \to \nu$ weak $*$ in $\mathcal M(\bbR_{\infty})$, that is,
\[
\lim_{n\to\infty} \int_{\R_{\infty}} f(t)\, d\nu_n(t) = \int_{\R_{\infty}} f(t)\, d\nu(t)
\]
for all $f\in C(\bbR_{\infty})$;\\
(c) $F_n \to F$ in value distribution, that is, \eqref{convvd} holds for all
Borel sets $A, S\subset\R$, $|A|<\infty$;\\
(d) \eqref{convvd} holds for all open, bounded intervals $A=(a,b)$, $S=(c,d)$.
\end{Theorem}
The implication (a) $\Longrightarrow$ (c) is an abstract version of
results from \cite{P}.
\begin{proof}
It is well known that (a) and (b) are equivalent, but we sketch the argument anyway: Observe
first of all that $F(i) =a+i\nu(\bbR_{\infty})$, so if (a) holds, then $a_n\to a$ and
the $\nu_n$ form a bounded sequence in $\mathcal M (\bbR_{\infty})$. By the Banach-Alaoglu Theorem,
we can extract a weak $*$ convergent subsequence $\nu_{n_j}\to\mu$.
We can then pass to the limit
in the Herglotz representations \eqref{Hrep} of the $F_{n_j}$
and use the uniqueness of such representations
to conclude that $\mu=\nu$. In particular, this is the only possible limit point of
the $\nu_n$ and thus it was not necessary to pass to a subsequence.

Conversely, if (b) holds, just pass to the limit in the Herglotz representations of the
$F_n$ to confirm pointwise convergence. A normal families argument then improves this to
locally uniform convergence, which is what (a) claims.

We now prove that (a) implies (c), following \cite{P}. A different proof could be based
on Lemma \ref{LA.1} below.

Given $F\in \mathcal H$, let
\[
F^{(y)}(z) = \frac{1+yF(z)}{y-F(z)} \quad\quad\quad (y\in\bbR_{\infty}) .
\]
It's easy to check that $F^{(y)}\in\mathcal H$; in fact,
\begin{equation}
\label{2.1}
\textrm{Im }F^{(y)}(z) = (1+y^2)\frac{\textrm{Im }F(z)}{|y-F(z)|^2} .
\end{equation}
The following formula (``spectral averaging'') will be crucial: If $A,S\subset\bbR$
are Borel sets, $|A|<\infty$, then
\begin{equation}
\label{spa}
\int_A \omega_{F(t)}(S)\, dt = \int_S \rho^{(y)}(A)\, \frac{dy}{1+y^2} .
\end{equation}
Here, $d\rho^{(y)}(t) = (1+t^2)\chi_{\bbR}(t)\, d\nu^{(y)}(t)$, and $\nu^{(y)}$ is the measure
from the Herglotz representation of $F^{(y)}$. See \cite[Theorem 1]{P}.

As a final preparation for the proof of (a) $\Longrightarrow$ (c), we now establish
part of the implication (d) $\Longrightarrow$ (c). Namely, we claim that if \eqref{convvd}
holds for all $A=(a,b)$ and fixed $S$, then it holds for \textit{all} Borel sets $A$ of finite
Lebesgue measure (and the same $S$). Let us prove this now:
fix $S$, and to simplify the notation, abbreviate
$\omega_{F_n(t)}(S) = \omega_n$, $\omega_{F(t)}(S) = \omega$. Suppose that \eqref{convvd}
holds for all $A=I=(a,b)$. Then, if we are given disjoint intervals $I_j$ with
$\left| \bigcup I_j \right| < \infty$, then, by dominated and monotone convergence,
\[
\int_{\bigcup I_j} \omega_n\, dt = \sum_j \int_{I_j} \omega_n \, dt \to
\sum_j \int_{I_j} \omega \, dt = \int_{\bigcup I_j} \omega\, dt
\]
(because $0\le \omega_n \le 1$, thus $0\le \int_{I_j} \omega_n \, dt \le |I_j|$, and
$\sum |I_j| < \infty$).

If now $A$ is an arbitrary Borel set of finite measure and $\epsilon>0$, we can find disjoint open intervals
$I_j$ (using the regularity of Lebesgue measure) so that
\[
A\subset \bigcup I_j, \quad\quad \left| \bigcup I_j \setminus A \right| < \epsilon .
\]
Then
\begin{gather*}
\int_{\bigcup I_j} \omega_n\, dt - \epsilon <
\int_A \omega_n\, dt \le \int_{\bigcup I_j} \omega_n \, dt , \\
\int_A \omega\, dt \le \int_{\bigcup I_j} \omega\, dt < \int_A \omega\, dt + \epsilon .
\end{gather*}
As noted above,
\[
\int_{\bigcup I_j} \omega_n\, dt \to \int_{\bigcup I_j} \omega\, dt ,
\]
so, putting things together, we see that $\liminf, \limsup \int_A \omega_n\, dt$ both differ from
$\int_A \omega\, dt$ by at most $\epsilon$, but $\epsilon>0$ was arbitrary, so we obtain that
\[
\int_A \omega_n \, dt \to \int_A \omega\, dt ,
\]
as desired.

Thus, returning to the proof of (a) $\Longrightarrow$ (c) now, we may assume that
$A=(a,b)$. Let
$R=\max (|a|,|b|)$. Then
\begin{equation}
\label{2.4}
\rho^{(y)}(A) \le (1+R^2) \nu^{(y)}(A) \le
(1+R^2) \nu^{(y)}(\bbR_{\infty}) = (1+R^2) \textrm{Im }F^{(y)}(i) .
\end{equation}
With $F$ replaced by $F_n$, this identity together with \eqref{2.1} show that
\begin{equation}
\label{2.2}
\rho_n^{(y)}(A) \le C \quad\quad (n\in\N, y\in\bbR_{\infty})
\end{equation}
(because $F_n(i)\to F(i)$, $\textrm{Im }F(i)>0$).

Since also $F_n^{(y)}\to F^{(y)}$, locally uniformly, we have the weak $*$ convergence of the
measures by the equivalence of (a) and (b) and thus
\begin{equation}
\label{2.3}
\rho_n^{(y)}(A)\to \rho^{(y)}(A) ,
\end{equation}
except possibly for two values of $y$ (for those values of $y$ for which $\rho^{(y)}(\{ a, b \}) \not= 0$).
Now we can use \eqref{spa} with $F$ replaced by $F_n$;
\eqref{2.2}, \eqref{2.3} show that the hypotheses of the dominated convergence theorem are satisfied,
so we can pass to the limit on the right-hand sides.

Finally, we show that (d) implies (a). This can be done very conveniently
using just compactness and uniqueness. More specifically,
pick a subsequence (denoted by $F_n$ again, to keep the notation
manageable) that converges locally uniformly to $G$ (possible by normal families). Here, either
$G\in\mathcal H$, or else $G\equiv a\in\bbR_{\infty}$. Actually, only the first case can occur here:
If, for instance, $F_n\to a\in\bbR$, then \eqref{2.1}, \eqref{2.4} show that for every $R>0$,
\[
\rho_n^{(y)}([-R,R])\to 0 \quad\quad (n\to \infty) ,
\]
uniformly in $|y-a|\ge \delta >0$. Therefore, by \eqref{spa},
\[
\int_{-R}^R \omega_{F_n(t)}((a-R,a-\delta)\cup(a+\delta,a+R))\, dt \to 0 \quad\quad (n\to\infty) .
\]
By hypothesis, we then also must have that
\[
\omega_{F(t)}((a-R,a-\delta)\cup(a+\delta,a+R))=0
\]
for almost every $t\in (-R,R)$. This is clearly not possible if $F(t)\equiv\lim F(t+iy)\in\C^+$,
and if $F(t)$ exists and is real, then, since $R,\delta>0$ are arbitrary, it follows that
$F(t)=a$. In other words, $F(t)=a$ almost everywhere,
but this is not a possible boundary value of an $F\in\mathcal H$.

A similar argument rules out the case where $|F_n|\to\infty$. In fact, we can also work with $G_n=-1/F_n$
and then run the exact same argument again.

Thus $F_n\to G\in\mathcal H$, uniformly on compact sets. But then, by the already established
implication (a) $\Longrightarrow$ (c), $F_n\to G$
in value distribution, and since such a limit
is unique, $G=F$. (What we actually use
here is the statement that $F_n$ can have at most one limit in the sense of (d), but the argument given
above, in the paragraph preceding Theorem \ref{T2.1}, establishes exactly this.)
Now every subsequence
of $\{ F_n \}$ has a locally uniformly convergent sub-subsequence, but, as we just saw, the corresponding
limit can only be $F$, so in fact $F_n\to F$ locally uniformly, without the need of passing to a
subsequence, and this is what (a) claims.
\end{proof}
A common thread in this proof was the following: because of compactness properties operating in the
background, convergence conditions are often self-improving. We give three more examples for this theme,
which can be extracted from the preceding proof.
\begin{Proposition}
\label{P2.2}
Let $F_n, F\in\mathcal H$.
The following conditions are also equivalent to the statements from Theorem \ref{T2.1}:\\
(a) $\lim_{n\to\infty} F_n(z_j) = F(z_j)$ on a set
$\{ z_j \}_{j=1}^{\infty}$ with a limit point in $\C^+$;\\
(b) There exists a Borel set $B\subset\bbR$ of finite positive Lebesgue measure so that for all Borel sets
$A\subset B$ and all bounded open intervals $J=(c,d)$, we have that
\[
\lim_{n\to\infty} \int_A \omega_{F_n(t)}(J)\, dt = \int_A \omega_{F(t)}(J)\, dt .
\]
(c) $F_n\to F$ in value distribution, and the convergence in \eqref{convvd} is uniform in $S$.
\end{Proposition}
\begin{proof}[Sketch of proof]
(a) Use the compactness property of $\mathcal H \cup \{ F\equiv a: a\in\R_{\infty} \}$
(normal families again!).

(b) Use compactness and recall that Herglotz
functions are uniquely determined by their boundary values on a set of positive Lebesgue measure.

(c) Return to the proof of Theorem \ref{T2.1}. The argument based on \eqref{2.4}, \eqref{2.2} shows
that the convergence is uniform in $S$ at least for $A=(a,b)$. Now we can again approximate
$A$ by disjoint unions of intervals, and this lets us extend the statement to arbitrary Borel sets $A$.
\end{proof}
\section{Proof of Proposition \ref{P1.1} and Theorem \ref{TBP}}
\begin{proof}[Proof of Proposition \ref{P1.1}]
These statements summarize some standard facts about $\omega$ limit sets;
see, for example, \cite{Kur,Wal} for background information. Extend $V\in\mathcal V_+^C$
to a whole line potential $V\in\mathcal V^C$, for example by letting $V(n)=0$
(Jacobi case: $V(n)=(a(n),b(n))=(1,0)$) for $n\le 0$. Then the representation
\[
\omega(V) = \bigcap_{m\ge 1} \overline{ \{S^n V : n\ge m \} }
\]
is valid, and this exhibits $\omega(V)$ as an intersection of a decreasing sequence
of compact sets. Thus $\omega(V)$ is non-empty and compact. It is also clear that
$\omega(V)$ is invariant under $S$ and $S^{-1}$, so $S$, restricted to $\omega(V)$, is
a homeomorphism. If the final claim of Proposition \ref{P1.1} were wrong, there would be
a subsequence $n_j\to\infty$ so that $d(S^{n_j}V, W)\ge\epsilon >0$ for all $j$
and all $W\in\omega(V)$. But by
compactness, $S^{n_j}V$ must approach a limit on a sub-subsequence, and this limit must
lie in $\omega(V)$. This is a contradiction.
\end{proof}
We now turn to proving Theorem \ref{TBP}. As already explained, this is a reformulation
of \cite[Theorem 1]{BP1}, so I state this result first. We follow \cite{BP1} and treat
half line problems (however, as we'll discuss, an analogous result for whole line problems
is also valid, and in fact this may be the more natural version because of the greater symmetry
of its setup). Given coefficients on $\Z_+$, we will cut this half line into two smaller intervals
at a variable point $n$. We then denote the $m$ functions
of the problems on $\{ 1,2, \ldots, n\}$ and $\{ n+1, n+2,\ldots \}$ by $m_-(n,\cdot)$ and
$m_+(n,\cdot)$, respectively; precise definitions will be given in a moment.
\begin{Theorem}[Breimesser-Pearson \cite{BP1}]
\label{TBPorig}
Consider a (half line) Jacobi matrix $J$ with bounded coefficients. For all
Borel sets $A\subset \Sigma_{ac}$, $S\subset\bbR$, we have that
\[
\lim_{n\to\infty} \left(
\int_A \omega_{m_-(n,t)}(-S)\, dt - \int_A \omega_{m_+(n,t)}(S)\, dt \right) = 0 .
\]
Moreover, the convergence is uniform in $S$.
\end{Theorem}
The \textit{moreover} part is not explicitly stated in \cite{BP1}, but, as we will show
in the appendix, it does follow from the proof that is given. It is probably quite useless anyway.

We now summarize some basic facts about the $m$ functions $m_{\pm}$, for a quick orientation.
Please see also Appendix~A for a more elegant treatment using linear fractional transformations.

The definition goes as follows: For $z\in\C^+$, let $f_{\pm}(\cdot , z)$ be solutions of
\begin{equation}
\label{je}
a(n)f(n+1)+a(n-1)f(n-1)+b(n)f(n)=zf(n)
\end{equation}
that are in the domain of $J$ near the right (respectively, left) endpoint.
More precisely, we demand that
\[
a(0)f_-(0,z)=0 , \quad f_+(\cdot, z) \in \ell_2(\Z_+) .
\]
These conditions determine $f_{\pm}$ up to multiplicative constants. We then define
\begin{equation}
\label{defm}
m_-(n,z) = \frac{f_-(n+1,z)}{a(n)f_-(n,z)}, \quad m_+(n,z) = -\frac{f_+(n+1,z)}{a(n)f_+(n,z)} .
\end{equation}
The lack of symmetry between $f_-$ and $f_+$ comes from the fact that we are considering half line problems,
and if a half line is cut into two parts, we obtain another half line and a finite interval
(not two half lines).
We could in fact pass to a more symmetric formulation of Theorem \ref{TBPorig} very easily:
we would then extend the coefficients to $\Z$ (for example, by putting $a(n)=1$, $b(n)=0$ for $n\le 0$)
and work with $f_- \in \ell_2(\Z_-)$ instead of the $f_-$ defined above. Theorem \ref{TBPorig}
holds in this situation as well, with an almost identical proof.

Let me repeat one important point just made:
we can also define $m$ functions $m_{\pm}$ for whole line operators if we
make the adjustment mentioned in the preceding paragraph: $f_-$ is now defined
by requiring that $f_-(\cdot, z)\in\ell_2(\Z_-)$.
In particular, the $m$ functions of a reflectionless whole line potential, as in Definition \ref{D1.1}, are
defined in this way, with $n=0$.
We will soon have occasion to apply these remarks again, when we prove Theorem \ref{TBP}.

The definition of $m_{\pm}$
shows that $m_+(n,\cdot)$ only depends on $a(j)$, $b(j)$ for $j>n$, while
$m_-(n,\cdot)$ only depends on the coefficients for $1\le j\le n$. So these functions refer to
disjoint subsets of $\Z_+$, and this observation immediately gives Theorem \ref{TBPorig}
a somewhat paradoxical flavor.

The functions $m_{\pm}(n, \cdot)$ are Herglotz functions, and they can be used in the usual
way to construct spectral representations.
Namely, we have that
\[
m_+(n,z) = \langle \delta_{n+1}, (J_n^+ - z)^{-1} \delta_{n+1} \rangle ,
\]
where $J_n^+$ is the Jacobi matrix, restricted to $\ell_2(\{ n+1, n+2,\ldots \} )$, and
$\delta_j$ denotes the unit vector located at $j$: $\delta_j(j)=1$, $\delta_j(k)=0$ for $k\not= j$.
This follows quickly by observing that
\[
f(j,z) = \langle \delta_j , (J_n^+ - z)^{-1} \delta_{n+1} \rangle
\]
solves \eqref{je} for $j> n+1$ and lies in $\ell_2$ and thus must be a multiple of $f_+(j,z)$ for
$j\ge n+1$. The Herglotz representation of $m_+(n,\cdot)$ thus reads
\begin{equation}
\label{3.2}
m_+(n,z) = \int_{-\infty}^{\infty} \frac{d\rho_n^+(t)}{t-z} ,
\end{equation}
where $d\rho_n^+$ is the spectral measure of $J_n^+$ and $\delta_{n+1}$.

A similar discussion applies to $m_-(n,\cdot)$. Note, however, that $m_-$ is not just the
mirror version of $m_+$: swapping left and right means that $n$ and $n+1$ should also change roles,
but there is no such change in \eqref{defm}. Rather, we have the following substitute for \eqref{3.2}:
\[
m_-(n,z) = \frac{z-b(n)}{a(n)^2} + \frac{a(n-1)^2}{a(n)^2}
\int_{-\infty}^{\infty} \frac{d\rho_n^-(t)}{t-z}
\]
Here, $d\rho_n^-$ is the spectral measure of the
restriction of $J$ to $\{ 1, \ldots, n-1 \}$ and $\delta_{n-1}$.
Put differently, $m_-(n, \cdot)$ is the $m$ function of the problem on $\{ 1, \ldots, n \}$
(but with the usual
roles of left and right interchanged) with Neumann boundary conditions at $n$ ($f(n)=0$) and the usual
(Dirichlet) boundary conditions at $1$ ($f(0)=0$).
See \cite[Chapter 2]{Teschl} for a more complete treatment of these issues (two warnings are in order:
what we called $m_{\pm}$ above is denoted by
$\widetilde{m}_{\pm}$ in \cite{Teschl}, and $b(n)$ in formula (2.15)
of \cite{Teschl} should read $b(0)$).

We will give a detailed proof of Theorem \ref{TBPorig} in Appendix A. Let us now show that
Theorem \ref{TBP} indeed follows from this result.

If $W\in\mathcal V^C$ is a whole line potential
(\textit{Jacobi reader: }recall that this term may well refer
to the coefficients of a whole line Jacobi operator),
we write $W_{\pm}$ for the restrictions of $W$ to $\Z_{\pm}$, and, as above, we denote
the set of these restrictions by
\[
\mathcal V_{\pm}^C = \left\{ W_{\pm}: W\in\mathcal V^C \right\} .
\]
As a final preparation, we recall a basic continuity property.
\begin{Lemma}
\label{L4.1}
The maps
\[
\mathcal V_{\pm}^C \to \mathcal H , \quad\quad W_{\pm} \mapsto M_{\pm} = m_{\pm}^W(0,\cdot )
\]
are homeomorphisms onto their images. (On $\mathcal H$, we use the topology
of uniform convergence on compact sets, or one of the equivalent descriptions
of this topology, as in Theorem \ref{T2.1}; note that this space is metrizable.)
\end{Lemma}
\begin{proof}
This is folk wisdom and can be seen in many ways. We will therefore only provide
sketches of possible arguments.

The correspondence $W_+ \leftrightarrow M_+$ is
one-to-one (of course, this also holds for $\Z_-$, but we will explicitly discuss
only the right half line here). See, for example, \cite{Simmom,Teschl}.

The continuity of the map $W_+\mapsto M_+$ can be conveniently deduced from the
basic constructions of Weyl theory. (Sketch: The coefficients on an initial interval
$\{1, \ldots, N\}$ determine a Weyl disk for
every fixed $z\in\C^+$, and $M_+(z)$ lies in that disk. Center and radii
of these disks depend continuously on
$W(1),\ldots, W(N)$, and the radii go to zero as $N\to\infty$, locally
uniformly on $\C^+$ because
we assumed that $W\in\mathcal V^C$ and this implies limit point case at $\infty$.)

Alternatively, one can use moments, as follows: The moments $\mu_j=\int x^j\, d\rho_+(x)$ for
$0\le j\le 2N$ are continuous functions of $W(1), \ldots , W(N)$ and by \eqref{3.2},
\[
M_+(w^{-1}) = - \sum_{j=0}^{\infty} \mu_j w^{j+1}
\]
on a disk about $w=0$.
These two facts readily imply that $M_+$ depends continuously on $W_+\in\mathcal V_+^C$.

Continuity of the inverse map $M_+\mapsto W_+$ is now actually automatic
(an invertible continuous map between compact metric spaces has a continuous inverse).

On top of that, it's also easy to give an honest proof of the continuity of the map
$M_+\mapsto W_+$. For instance,
one can argue as follows: If $M_+^{(j)}$ and $M_+$ are the
$m$ functions of certain potentials $W_+^{(j)}, W_+\in\mathcal V_+^C$,
and $M_+^{(j)}\to M_+$, uniformly on compact subsets of $\C^+$, then, by Theorem \ref{T2.1}, the
spectral measures $\rho_+^{(j)}$ converge to $\rho_+$ in weak $*$ sense.
Since we are dealing with coefficients satisfying uniform bounds,
the supports are all contained in a fixed bounded set $[-R,R]$. It follows that the moments
$\int x^n \, d\rho_+^{(j)}(x)$ converge, too, and this implies convergence of the coefficients $W_+^{(j)}$
(for example, because there are explicit formulae that recover these coefficients from the moments;
see \cite{Simmom} or \cite[Sect.\ 2.5]{Teschl}).
\end{proof}
\begin{proof}[Proof of Theorem \ref{TBP}]
Let $W\in\omega(V)$ (in particular, $W$ is a whole line potential).
Then there exists $n_j\to\infty$ so that $d(S^{n_j}V,W)\to 0$.
By Lemma \ref{L4.1}, we then have that
\[
m_{\pm}(n_j,z) \to M_{\pm}(z) \quad\quad (j\to\infty) ,
\]
uniformly on compact subsets of $\C^+$. Here, $M_{\pm}(z)=m_{\pm}^W(0,z)$
are the $m$ functions of the (whole line) potential $W$.
Note that $m_-(n_j,z)$ lies in the Weyl disk for $(S^{n_j}V)_-$, so the argument
from the proof of Lemma \ref{L4.1} does work and the fact that
this $m$ function does not refer to a full half line
does not cause any problems.

Theorem \ref{TBPorig}, if combined with Theorem \ref{T2.1}, now says that
\[
\int_A \omega_{M_-(t)}(-S) \, dt = \int_A \omega_{M_+(t)}(S)\, dt
\]
for all Borel sets $A\subset\Sigma_{ac}$, $S\subset\R$.
Now the argument presented in Sect.~2 in the paragraph preceding Theorem \ref{T2.1}
concludes the proof:
By Lebesgue's differentiation theorem,
\begin{equation}
\label{3.3}
\omega_{M_-(t)}(-S) = \omega_{M_+(t)}(S)
\end{equation}
for $t\in\Sigma_{ac}\setminus N$, $|N|=0$, and all intervals $S$ with rational endpoints
(or other countable collections of sets $S$).
We can also assume that $M_{\pm}(t)=\lim_{y\to 0+} M_{\pm}(t+iy)$ exist for these $t$.
If $M_-(t)\in\R$, then, by choosing small intervals about this value for $-S$, we see
that $M_+(t)=-M_-(t)$. If $M_-(t)\in\C^+$, then, as explained in Sect.~2, we can define
$\omega_{M_-(t)}$ directly, using \eqref{harmmeas} rather than a limit. This formula also shows that
then $\omega_{M_-(t)}(-S) = \omega_{-\overline{M_-(t)}}(S)$, and \eqref{3.3} now implies that
\begin{equation}
\label{3.4}
M_+(t) = -\overline{M_-(t)} .
\end{equation}
This is also what we found in the other case ($M_-(t)\in\R$), so \eqref{3.4} in fact holds
for almost every $t\in\Sigma_{ac}$, that is, $W\in\mathcal R (\Sigma_{ac})$, as claimed.
\end{proof}
It will be convenient to extract, for later use, a technical fact from this proof.
\begin{Lemma}
\label{L3.1}
Let $W$ be a bounded (whole line) potential, and, as above, denote its $m$ functions
by $M_{\pm}(z)=m_{\pm}^W(0,z)$. Then
$W\in \mathcal R (A)$ if and only if
\begin{equation}
\label{3.1}
\int_B \omega_{M_-(t)}(-S) \, dt = \int_B \omega_{M_+(t)}(S)\, dt
\end{equation}
for all Borel sets $B\subset A$, $|B|<\infty$, $S\subset\R$.
\end{Lemma}
\begin{proof}
It was proved above that \eqref{3.1} implies that $W\in\mathcal R (A)$,
the crucial ingredient being the trivial observation that $\omega_z(-S)=
\omega_{-\overline{z}}(S)$ for arbitrary $z\in\C^+$.

The converse follows just as quickly: It is well known that if $W\in\mathcal R (A)$, then
$\textrm{Im }M_{\pm}(t)>0$ for almost every $t\in A$. We will prove this again in
Proposition \ref{P4.1}(b) below. This, however, means that we can use the direct definition
of $\omega$ in \eqref{3.1}. In other words, we just substitute $M_{\pm}$ for $z$ in \eqref{harmmeas}.
But then \eqref{3.1} simply follows from the observation just discussed because
$-\overline{M_-}=M_+$ almost everywhere on $A$ by assumption.
\end{proof}
\section{Proof of the Oracle Theorem}
Given Theorem \ref{TBP}, the proof of Theorem \ref{Tor} consists essentially of an
adaptation of arguments of Kotani \cite{Kotac,Kot}. In a nutshell, the argument runs as follows:
Write again $W_{\pm}$ for the restrictions of a potential $W$ to $\Z_{\pm}$. By Proposition \ref{P1.1}
and Theorem \ref{TBP}, for large $n$, the distance
$d(S^n V, W)$ is small for a suitable
$W\in\mathcal R (\Sigma_{ac})\subset\mathcal R (A)$ ($W$ of course depends on $n$).
It is a well known fact (and will be discussed below again, see Proposition \ref{P4.1}(c)) that
if $W\in\mathcal R(A)$, then
$W_-$ determines $W_+$ (and vice versa). Now $S^nV$ is close to $W$, so if we can establish the continuity
of all the maps involved, then it should also be true that approximate knowledge of
$(S^nV)_-$ approximately determines $(S^nV)_+$.

Let us now look at the details of this argument.
We first collect some basic facts about reflectionless potentials.
We will write $\mathcal R_{\pm}(A)$ for the set of restrictions
$W_{\pm}$ of potentials $W\in\mathcal R (A)$
to $\Z_{\pm}$.
\begin{Proposition}
\label{P4.1}
Suppose that $A\subset\R$, $|A|>0$, and let $W\in\mathcal R(A)$.\\
(a) $M_+(n,t) = - \overline{M_-(n,t)}$ for almost every $t\in A$
for all $n\in\Z$. In other words, $S^nW\in\mathcal R(A)$ for all $n\in\Z$;\\
(b) The two half line operators satisfy $\Sigma_{ac}(W_{\pm}) \supset A$;\\
(c) $W_-$ (or $W_+$) uniquely determines $W$. Put differently,
the restriction maps
\[
\mathcal R(A)\to \mathcal R_{\pm}(A), \quad\quad W\mapsto W_{\pm}
\]
are injective;\\
(d) For every $C>0$, $\mathcal R^C(A):=\mathcal R(A) \cap \mathcal V^C$ is compact;\\
(e) Define $\mathcal R_{\pm}^C(A):= \{ W_{\pm} : W\in\mathcal R^C(A)\}$. Then the map
\[
\mathcal R_-^C(A) \to \mathcal R_+^C(A), \quad\quad W_-\mapsto W_+
\]
(which is well defined, by part (c)) is uniformly continuous.
\end{Proposition}
\begin{proof}
(a) For $n=0$, this is the definition of $\mathcal R (A)$, and for arbitrary $n$, it follows
from the evolution of $M_{\pm}$ (Riccati equation; compare \cite[Lemma 8.1]{Teschl}, or see
the discussion in the appendix to this paper).

(b) From the definition of $\mathcal R(A)$, we have that $\textrm{Im }M_-=\textrm{Im }M_+$
almost everywhere on $A$. If we had $\textrm{Im }M_{\pm}(t)=0$ on a subset of $A$ of
positive measure,
then it would follow that $M_-(t)+M_+(t) =0$ on this set, but $M_-+M_+$ is a Herglotz function
and these are determined by their boundary values on any set of positive measure, so it would
actually follow that $M_-(z)+M_+(z)\equiv 0$, which is absurd.

(c) $W_-$ determines $M_-$, and if $W\in\mathcal R (A)$,
then $M_-$ determines $M_+$ almost everywhere on $A$.
As just discussed, this determines $M_+(z)$ completely, and from $M_+$ we can go back to $W_+$.

(d) Since $\mathcal V^C$ is compact, it suffices to show that $\mathcal R^C(A)$ is closed.
So assume that $W_j\in \mathcal R^C(A)$, $W\in\mathcal V^C$, $W_j\to W$. By Lemma \ref{L4.1}
and Theorem \ref{T2.1}, we then have convergence in value distribution of the $m$ functions
$M_{\pm}^{(j)}$ to $M_{\pm}$. By Lemma \ref{L3.1}, this implies that
\[
\int_B \omega_{M_-(t)}(-S)\, dt = \int_B \omega_{M_+(t)}(S)\, dt
\]
for all Borel sets $B\subset A$, $S\subset\R$,
so the claim now follows by again using Lemma \ref{L3.1} (the other direction this time).

(e) This follows at once from (c) and (d): If restricted to $\mathcal R^C(A)$, the map
$W\mapsto W_-$ is an injective, continuous map between compact metric spaces. Therefore,
its inverse
\[
\mathcal R_-^C(A) \to \mathcal R^C(A) , \quad\quad W_-\mapsto W
\]
is continuous, too. Thus the association $W_-\mapsto W_+$ is the composition of the
continuous maps $W_-\mapsto W\mapsto W_+$. \textit{Uniform} continuity is automatic
because $\mathcal R_-^C(A)$ is compact.
\end{proof}
There are some pitfalls hidden here for the over-zealous. For example, it is not
true in general that $\Sigma_{ac}(W_-)=\Sigma_{ac}(W_+)$ if $W$ is reflectionless.
Perhaps somewhat more disturbingly, it is also not true in general that we can recover
$W$ from $W_-$ if we only know that $W$ is reflectionless on \textit{some} set.
In other words, there exist potentials $W^{(j)}\in\mathcal R(A_j)$ ($j=1,2$), so that
$W_-^{(1)}=W_-^{(2)}$, but $W^{(1)}\not= W^{(2)}$. We will return to these issues
in Sect.~5. See especially Theorems \ref{T5.1}, \ref{T5.2}.

Let us now use Proposition \ref{P4.1} to prove Theorem \ref{Tor}.
\begin{proof}[Proof of the Oracle Theorem]
Let $A\subset\R$, $C>0$, $\epsilon>0$ be given. Determine $\delta>0$ so that
\begin{equation}
\label{4.9}
\left| W(1)-\widetilde{W}(1)\right| < \epsilon \quad\quad\textrm{ if }\:
W,\widetilde{W}\in \mathcal R^C(A),\quad d(W_-, \widetilde{W}_-)< 5\delta .
\end{equation}
This is possible by Proposition \ref{P4.1}(e). For technical reasons, we also demand
that $\delta<\epsilon$.

Next, consider the (closed) $2\delta$-neighborhood of $\mathcal R_-^C(A)$. We will write
\[
\mathcal U_{2\delta} =
\left\{ U_-\in\mathcal V_-^C: d\left( U_-,\mathcal R_-^C(A)\right) \le 2\delta \right\}
\]
for this set. Proposition \ref{P4.1} implies that
$\mathcal U_{2\delta}\subset \mathcal V_-^C$ is compact,
so we can cover this set by finitely many balls about elements
of $\mathcal R_-^C(A)$ as follows:
\[
\mathcal U_{2\delta} \subset B_{3\delta}(W_-^{(1)}) \cup \ldots \cup B_{3\delta}(W_-^{(M)}) ,
\]
where $W_-^{(j)}\in\mathcal R_-^C(A)$.

We can now give a preliminary definition of the oracle $\Delta$ (smoothness will have to
be addressed later). Pick $L$ (sufficiently large) so that
$(3C+1) \sum_{j>L} 2^{-j} < \delta$. This choice of $L$ makes sure that $d(U_-,\widetilde{U}_-)<\delta$
whenever $U(n)=\widetilde{U}(n)$ for $n=0,-1,\ldots, -L$ (and $U_-,\widetilde{U}_-\in\mathcal V_-^C$).

\textit{Apology to the Jacobi reader: }While I usually try to write things up in such a way that the
standard replacement $V\to (a,b)$ is the only adjustment that has to be made, this is unfortunately
not the case in this last part of this proof. Here, more extensive changes in the notation
(but not in the underlying argument, which remains valid) become necessary.

If $u_{-L}, \ldots , u_0$ are given numbers with $|u_j|\le C$ which have the property that
\[
[ \ldots, 0,0,0, \ldots, 0, u_{-L}, \ldots, u_0 ] \in B_{3\delta}(W_-^{(1)}) ,
\]
then put
\[
\Delta(u_{-L}, \ldots , u_0 ) = W^{(1)}(1) .
\]
Having done that, move on to the next ball: If
\[
[ \ldots, 0,0,0, \ldots, 0, u_{-L}, \ldots, u_0 ] \in B_{3\delta}(W_-^{(2)})\setminus
B_{3\delta}(W_-^{(1)}) ,
\]
define
\[
\Delta(u_{-L}, \ldots , u_0 ) = W^{(2)}(1) .
\]
Continue in this way. It could happen that,
after having dealt with the last ball $B_{3\delta}(W_-^{(M)})$,
there are still points left in $[-C,C]^{L+1}$ for which $\Delta$ has not yet been defined. However,
by the construction of the balls, these are points that can never be close to any reflectionless
potential, so they are irrelevant as far as
Theorem \ref{Tor} is concerned (because by Proposition \ref{P1.1}
and Theorem \ref{TBP},
$(S^nV)_-$ \textit{will} eventually be close to certain reflectionless potentials).
If a complete (preliminary) definition
of the function $\Delta$ is desired,
we can assign arbitrarily chosen values to these points ($\Delta=0$, say).

It remains to show that $\Delta$ indeed predicts $V(n+1)$. To this end, take $n_0\ge L$ so large that
$d(S^nV, \omega(V))<\delta$ for $n\ge n_0$. This is possible by Proposition \ref{P1.1}.
So, by Theorem \ref{TBP}, for every $n\ge n_0$, we can find $\widetilde{W}\in \mathcal R^C(A)$ so that
\begin{equation}
\label{4.1}
d\left( S^nV, \widetilde{W}\right) < \delta .
\end{equation}
Note that $\widetilde{W}$ will usually depend on $n$, but $n$ is fixed in this part of the argument,
so we suppress this dependence in the notation.
By the choice of $L$, \eqref{4.1} clearly implies that
\[
[ \ldots, 0,0,0, \ldots, 0, (S^nV)(-L), \ldots, (S^nV)(0) ] \in \mathcal U_{2\delta} .
\]
Thus there exists $j\in \{ 1, \ldots, M\}$ so that
\[
[ \ldots, 0,0,0, \ldots, 0, V(n-L), \ldots, V(n) ] \in B_{3\delta}(W_-^{(j)}) .
\]
Fix the minimal $j$ with this property. With this choice of $j$, we have that
\begin{equation}
\label{4.2}
\Delta(V(n-L),\ldots , V(n)) = W^{(j)}(1) ,
\end{equation}
by the construction of $\Delta$. Also,
\begin{align*}
d(W_-^{(j)},\widetilde{W}_-) & \le d(W_-^{(j)}, [\ldots,0,V(n-L),\ldots, V(n)]) + \\
& \quad \quad d([\ldots,0,V(n-L),\ldots, V(n)], (S^nV)_-) +
d((S^nV)_-, \widetilde{W}_-) \\
& < 3\delta + \delta + \delta = 5\delta ,
\end{align*}
thus, by the defining property \eqref{4.9} of $\delta$,
\begin{equation}
\label{4.3}
\left| W^{(j)}(1) - \widetilde{W}(1) \right| < \epsilon .
\end{equation}
On the other hand, \eqref{4.1} clearly also says that
\[
\left| V(n+1) - \widetilde{W}(1) \right| < 2\delta < 2\epsilon ,
\]
and if this is combined with \eqref{4.2}, \eqref{4.3}, we obtain that
\[
\left| V(n+1) - \Delta(V(n-L), \ldots , V(n))\right| < 3\epsilon ,
\]
as required.

The $\Delta$ constructed above is not continuous, but this is easy to fix.
We now sketch how this can be done. Note that $\Delta$ does have redeeming properties: it takes
only finitely many values and we can also make sure that there
exist a set $D\subset [-C,C]^{L+1}$ and $\delta_0>0$ so that $|\Delta(x)-\Delta(y)|<\epsilon$
whenever $x,y\in D$, $|x-y|<\delta_0$ (just replace $5\delta$ by $6\delta$ in \eqref{4.9}).
Moreover, we can redefine $\Delta$ on the complement of $D$
without affecting the statement of Theorem \ref{Tor}.
Therefore, by taking convolutions with suitable functions, we can
pass to a $C^{\infty}$ modification of $\Delta$
that still predicts $V(n+1)$ with accuracy $4\epsilon$, say.
\end{proof}
\section{More on reflectionless potentials}
We start out with some quick observations.
If $M_{\pm}$ are the $m$ functions of some $W\in\mathcal R(A)$,
then $H=M_+ +M_-$ is another Herglotz function and $\textrm{Re }H(t)=0$ for almost every $t\in A$.
We are therefore led to also consider these Herglotz functions,
in addition to the $m$ functions of reflectionless
potentials. We introduce
\begin{align*}
\mathcal N (A) & =
\left\{ H\in \mathcal H : \textrm{ Re }H(t)=0 \textrm{ for a.e.\ } t\in A \right\} , \\
\mathcal Q (A) & =
\left\{ F_+ \in\mathcal H :\textrm{ There exists }F_-\in\mathcal H \textrm{ so that }
F_+(t)=-\overline{F_-(t)}
\textrm{ for a.e.\ }t\in A\right\} .
\end{align*}
Note that if $F_+\in\mathcal Q (A)$,
then the $F_-$ from the definition is unique and $F_-\in \mathcal Q (A)$, too.

It is easy to determine all decompositions of the type $H=F_++F_-$, where $F_{\pm}\in\mathcal Q(A)$
are as above, of a given $H\in\mathcal N(A)$.
\begin{Proposition}
\label{P5.1}
Let $A\subset\bbR$, $|A|>0$, and suppose that $H\in\mathcal N (A)$.
Write $\nu\in\mathcal M (\bbR_{\infty})$
for the measure from the representation \eqref{Hrep} of $H$. Let $F_+\in\mathcal H$, and put $F_-=H-F_+$.
Then the following statements are equivalent:\\
(a) $F_{\pm}\in\mathcal Q (A)$ and $F_+=-\overline{F_-}$ almost everywhere on $A$;\\
(b) $F_+$ is of the form
\[
F_+(z) = a_+ + \int_{\bbR_{\infty}} \frac{1+tz}{t-z}\, f(t)\, d\nu(t) ,
\]
with $a_+\in\bbR$, $f\in L_1(\bbR_{\infty}, d\nu)$,
$0\le f\le 1$, $f=1/2$ (Lebesgue) almost everywhere on $A$.
\end{Proposition}
\begin{proof}
If $F_{\pm}$ have the properties stated in part (a), then,
by the uniqueness of the Herglotz representations,
the measures from \eqref{Hrep}
must satisfy $\nu_++\nu_-=\nu$.
So, in particular, $\nu_+\le \nu$, and we can write
\[
d\nu_+(t) = f(t)\, d\nu(t) , \quad\quad d\nu_-(t) = (1-f(t))\, d\nu(t)
\]
for some $f\in L_1(\R_{\infty}, d\nu)$, $0\le f\le 1$. For almost every $t\in A$, we have that
$\textrm{Im }H(t)>0$ (because otherwise $H(t)=0$ on a set of positive measure, which is impossible) and
\[
\textrm{Im }F_+(t) = f(t)\textrm{Im }H(t), \quad\quad
\textrm{Im }F_-(t) = (1-f(t))\textrm{Im }H(t) .
\]
In this context, recall that the imaginary
part of the boundary value of a Herglotz function equals $\pi$ times the density of the
absolutely continuous part of the associated measure $d\rho =\chi_{\R}(1+t^2)\, d\nu$.
It follows that $f=1/2$ Lebesgue almost everywhere on $A$.

Conversely, if $F_+$ is of the form described in part (b), then
\[
F_-(z) = a-a_+ + \int_{\bbR_{\infty}} \frac{1+tz}{t-z}\, (1-f(t))\, d\nu(t) .
\]
First of all, this shows that $F_-\in\mathcal H$. Since $f=1/2$ almost everywhere on $A$, we have
that $\textrm{Im }F_+(t)=\textrm{Im }F_-(t)$ for almost every $t\in A$. Moreover,
\[
\textrm{Re }F_+(t)+\textrm{Re }F_-(t) = \textrm{Re }H(t) = 0
\]
for almost every $t\in A$, thus indeed $F_+=-\overline{F_-}$ almost everywhere on $A$, as required.
\end{proof}
We are of course particularly interested in functions $H$ and $F_{\pm}$ that come
from Jacobi matrices. More specifically, we want to start out with an $H\in\mathcal N(A)$ and then
find all $W\in\mathcal R(A)$ corresponding to this $H$.
This will be achieved in Corollary \ref{C5.1} below.

We will need an inverse spectral theorem. We deal with this issue first and incorporate the additional
conditions imposed by the requirement that $W\in\mathcal R(A)$ afterwards.
The fundamental result in this context says that any probability
measure on the Borel sets of $\R$ with bounded, infinite support is the spectral measure
of a unique (half line) Jacobi matrix on $\Z_+$ with $0<a(n)\le C$, $|b(n)|\le C$ for
some $C>0$, but in this form, the result is not immediately useful here. Rather,
the following version is tailor made for our needs.
\begin{Theorem}
\label{Tinv}
Let $H\in\mathcal H$. There exist a (whole line)
Jacobi matrix $J$ with bounded coefficients ($0<a(n)\le C$, $|b(n)|\le C$)
and a constant $c>0$
so that $cH=M_++M_-$ if and only if $H$ is of the form
\begin{equation}
\label{5.8}
H(z) = A + Bz + \int_{\R} \frac{d\rho(t)}{t-z} ,
\end{equation}
with $B>0$, and $\rho$ is a finite measure on the Borel sets of $\R$
with bounded, infinite support.

If $H\in\mathcal H$ satisfies these conditions and if $F_{\pm}\in\mathcal H$ are such that
$F_++F_-=H$, then there exists a (whole line) Jacobi matrix (with bounded coefficients, as above)
so that $cF_{\pm}=M_{\pm}$ for some $c>0$ if and only if $F_+$ is of the following form:
\[
F_+(z) = \int_{\R} \frac{d\rho_+(t)}{t-z} ,
\]
and both $\rho_+$ and $\rho_-:=\rho-\rho_+$ have infinite supports.

In this case, the constant
$c>0$ is uniquely determined, and thus the pair $(H,F_+)$ completely determines the
Jacobi matrix. We have that $c=\rho_+(\R)^{-1}$. In particular, $c$ only depends on $F_+$.

Finally, the following formula holds:
\[
-\frac{B}{H(z)} = \langle \delta_0 , (J-z)^{-1} \delta_0 \rangle .
\]
\end{Theorem}
It is not hard, if somewhat tedious, to extract this result
from some standard material, which is presented, for example, in \cite[Sect.~2.1, 2.5]{Teschl}.
So we will not prove this here.

As already pointed out at the beginning of Sect.~2, the measures $\rho$ are related
to the measures $\nu$ from Proposition \ref{P5.1} by $d\rho = \chi_{\R}(1+t^2)\, d\nu$.
In particular, in terms of $\nu$, condition \eqref{5.8} says that $\nu(\{\infty\} )>0$ and
$\int_{\R} (1+t^2)\, d\nu(t) < \infty$ (and, as always, $\chi_{\R}\, d\nu$ must have bounded,
infinite support).

The following combination of Proposition \ref{P5.1} and Theorem \ref{Tinv} will
be particularly interesting for us here. It determines all $W\in\mathcal R(A)$ that are
associated with a given $H\in\mathcal N(A)$.
\begin{Corollary}
\label{C5.1}
Let $H\in\mathcal H$ satisfy the conditions from Theorem \ref{Tinv}, and assume that,
in addition, $H\in\mathcal N(A)$. Let $H=F_++F_-$ be a decomposition as in Proposition \ref{P5.1},
with $F_{\pm}\in\mathcal Q(A)$ and $F_+=-\overline{F_-}$ almost everywhere on $A$.
Furthermore, assume that the Herglotz representation of $F_+$
can be written in the form
\begin{equation}
\label{5.7}
F_+(z) = \int_{\R} \frac{f(t)\, d\rho(t)}{t-z} ,
\end{equation}
where $f$ has the same meaning as in Proposition \ref{P5.1}, and $\rho$ is the measure from
representation \eqref{5.8} of $H$.

Then there exists a unique $c>0$ so that $cF_{\pm}=M_{\pm}$ are the $m$ functions of a unique
reflectionless potential $W\in\mathcal R(A)$.

Conversely, any $W\in\mathcal R(A)$ for which $M_+^{(W)}+M_-^{(W)}=cH$ for some $c>0$ arises in this way.
\end{Corollary}
\begin{proof}
This follows at once by combining Proposition \ref{P5.1} and Theorem \ref{Tinv} if the following
quick observations are made: First of all,
the conditions about the measures involved having infinite support
are automatically satisfied because by the reflectionless condition, the absolutely continuous part
of $\nu$ on $A$ is equivalent to $\chi_A\, dt$ and $f=1/2$ almost everywhere on $A$. It is also
useful to recall in this context that if $W\in\mathcal R(A)$ for some positive measure set $A\subset\R$,
then we will automatically obtain the inequality $a(n)\ge\alpha>0$ from the fact that $W$
has non-empty absolutely continuous spectrum \cite{Dom,SimSp}.

As for the converse, note that if
$cH= M_++M_-$ for some $c>0$ and $m$ functions $M_{\pm}$ of a potential $W\in\mathcal R(A)$, then
$M_{\pm}\in\mathcal Q(A)$, and we then of course also
have that $H=F_++F_-$, with $F_{\pm}=c^{-1}M_{\pm}\in\mathcal Q(A)$, so we may as well start out with
decomposing $H$ as in Proposition \ref{P5.1}. This, however, forces us to run through the construction
just discussed, so there are no additional reflectionless potentials corresponding to $H$ that might
have been overlooked.
\end{proof}
The previous results are rather baroque in appearance, but, fortunately, the final conclusion is
transparent again. Indeed, the gist of the preceding discussion is contained
in the following recipe: Start with
an $H\in\mathcal N(A)$ of the form \eqref{5.8}. Then, the $W\in\mathcal R(A)$ associated with
this $H$ are in one-to-one correspondence with the functions $f\in L_1(\R, d\rho)$ satisfying
$0\le f\le 1$, $f=1/2$ (Lebesgue) almost everywhere on $A$.

We can now clarify two points that were raised in Sect.~1.3 and in the comment
following Proposition \ref{P4.1}.
\begin{Theorem}
\label{T5.1}
There exists $W\in\mathcal R(A)$ with $\Sigma_{ac}(W_-)\not=\Sigma_{ac}(W_+)$.
\end{Theorem}
\begin{proof}
Given the previous work, this is very easy:
Fix an $H \in\mathcal N(B)$ that also satisfies the conditions from Theorem \ref{Tinv}.
$H(z)=(z^2-4)^{1/2}$ would be one example (among many) for such an $H$; here, we can let $B=[-2,2]$.
Fix a subset $A\subset B$ with $|A|>0$,
$|B\setminus A|>0$, and let
\[
f = \frac{1}{2}\, \chi_A ,
\]
and define $F_+$ as \eqref{5.7}, using this $f$.
By Proposition \ref{P5.1} and Corollary \ref{C5.1},
the function $F_+\in\mathcal Q(A)$ corresponds to (unique) reflectionless
Jacobi coefficients $W\in\mathcal R(A)$. Since $1-f>0$ on all of $B$, but $f>0$ only on $A$,
we have that
\[
\Sigma_{ac}(W_+)=A, \quad\quad \Sigma_{ac}(W_-)\supset B .
\]
\end{proof}
When dealing with functions from $\mathcal N(A)$, the exponential Herglotz representation is
a very useful tool.
Therefore, we now quickly review some basic facts; see \cite{ArD1,ArD2} for a (much) more
detailed treatment of this topic.

First of all, if $H\in\mathcal H$, we can take a holomorphic logarithm, and if we choose the
branch with $0<\textrm{Im}(\ln H)<\pi$ (say), we again obtain a Herglotz function. Moreover,
since $\textrm{Im}(\ln H)$ is bounded, the measure from the Herglotz representation is purely
absolutely continuous. Thus we can recover $H$ from $\textrm{Im}(\ln H(t))$, up to a multiplicative
constant. More specifically,
given $H\in\mathcal H$, we can define
\begin{equation}
\label{5.5}
\xi(t) = \frac{1}{\pi} \lim_{y\to 0+} \textrm{Im} \left( \ln H(t+iy) \right) .
\end{equation}
The limit exists almost everywhere and $0\le \xi(t)\le 1$. We have that
\[
H(z) = |H(i)| \exp \left( \int_{-\infty}^{\infty} \left( \frac{1}{t-z} - \frac{t}{t^2+1}
\right) \xi(t)\, dt \right) .
\]
\begin{Proposition}
\label{P5.2}
Let $H\in\mathcal H$. Then
$H\in\mathcal N(A)$ if and only if $\xi(t)=1/2$ for almost every $t\in A$.
\end{Proposition}
This is obvious from the definition of $\xi$, but it is also exceedingly useful because
it expresses the condition of belonging to $\mathcal N(A)$ as a \textit{local }condition
on the \textit{imaginary }part of a (new) Herglotz function.
The original requirement that $\textrm{Re }H=0$
refers to the Hilbert transform of $\textrm{Im }H$ and thus is not local.
\begin{Theorem}
\label{T5.2}
There exist potentials $W^{(j)}\in\mathcal R(A_j)$ ($j=1,2$) so that
$W_+^{(1)}=W_+^{(2)}$, but $W^{(1)}\not= W^{(2)}$.
\end{Theorem}
\begin{proof}
We will again work with the Herglotz functions. Consider the following pair of functions:
\begin{align*}
H_1(z) & = (z+1)^{1/2} z^{\epsilon} (z-1)^{1/2-\epsilon} , \\
H_2(z) & = (z+1)^{1/2-\epsilon} z^{\epsilon} (z-1)^{1/2} .
\end{align*}
Here, $\epsilon>0$ is small, and powers of the type $w^{\alpha}$ with $w\in\C^+$ and $\alpha>0$
are defined as $w^{\alpha}=e^{\alpha\ln w}$, with $0<\textrm{Im}(\ln w) < \pi$.

If written in this way, it is not completely obvious that $H_j$ maps $\C^+$ to $\C^+$ again,
and the motivation for these particular choices also remains mysterious. Things become clear,
however, if we work with the exponential Herglotz representation.

If we let
\begin{align*}
\xi_1 & = \chi_{(-\infty,-1)} + \frac{1}{2}\chi_{(-1,0)} +
\left( \frac{1}{2} - \epsilon \right) \chi_{(0,1)} , \\
\xi_2 & = \chi_{(-\infty,-1)} + \left( \frac{1}{2}+\epsilon \right)
\chi_{(-1,0)} + \frac{1}{2} \chi_{(0,1)} ,
\end{align*}
and choose the multiplicative constants so that $H_j(z)=z +O(1)$ as $|z|\to\infty$,
then we obtain the functions $H_1$, $H_2$ introduced above.
In particular, we now see that indeed $H_j\in\mathcal H$.
Proposition \ref{P5.2} shows that
\[
H_j\in\mathcal N(I_j) , \quad\quad I_1=(-1,0), \quad I_2=(0,1)
\]
(this of course is also immediate from the original definition of $H_j$),
and $\xi_1$ on $I_2$ (and $\xi_2$ on $I_1$)
are small perturbations of the value $1/2$ that would correspond to a reflectionless $\xi$.

The crucial fact about $H_1$, $H_2$ is the following:
If $0<t<1$ and $0<\epsilon< \frac{1}{\pi}\arccos(1/2)$, then
\begin{equation}
\label{5.3}
\frac{\textrm{Im }H_2(t)}{\textrm{Im }H_1(t)} =
\frac{\textrm{Im }H_1(-t)}{\textrm{Im }H_2(-t)} \le 2 .
\end{equation}
To prove \eqref{5.3}, just compute the
ratios, using the definitions of $H_1$, $H_2$. Indeed, a straightforward calculation reveals that
these are equal to (one another and)
\[
\left( \frac{1-t}{1+t} \right)^{\epsilon} \frac{1}{\cos \epsilon\pi} \le \frac{1}{\cos \epsilon\pi} .
\]

Define $F_+\in\mathcal H$ as follows:
\begin{equation}
\label{5.4}
F_+(z) = \frac{1}{2\pi} \int_{I_1} \frac{\textrm{Im }H_1(t)\, dt}{t-z} +
\frac{1}{2\pi} \int_{I_2} \frac{\textrm{Im }H_2(t)\, dt}{t-z} .
\end{equation}
Now \eqref{5.3} shows that we can write $F_+$ as
\[
F_+(z) = \frac{1}{\pi} \int_{\R} \frac{f_j(t)\textrm{Im }H_j(t)\, dt}{t-z}
\]
for $j=1,2$, and in both cases $f_j$ satisfies $0\le f_j\le 1$, $f_j=1/2$ on $I_j$.
Indeed, from a comparison with \eqref{5.4}, we learn that
\[
f_1(t) = \begin{cases} 1/2 & t\in I_1 \\
\textrm{Im }H_2(t)/(2\,\textrm{Im }H_1(t)) & t\in I_2 \end{cases} ,
\]
and \eqref{5.3} ensures that $f_1\le 1$. Of course, a similar argument works for $f_2$.
Note also that the measures $\rho_j$ associated with $H_j$ are purely absolutely continuous
and supported by $(-1,1)$.

We have thus verified that both $(H_1,F_+)$ and $(H_2,F_+)$ satisfy the assumptions of Corollary \ref{C5.1}
(with $A=I_1$ and $A=I_2$, respectively).
Therefore, there exist potentials $W^{(j)}\in\mathcal R(I_j)$ corresponding
to these data. Since both potentials have the same positive half line $m$ function ($M_+^{(j)}=cF_+$),
it follows that $W_+^{(1)}=W_+^{(2)}$. On the other hand, the whole line potentials can not be identical
because $H_1\not= H_2$.
\end{proof}
\section{Small limit sets}
Recall the general philosophy behind Theorem \ref{TDR}: We assume that
\[
\sigma_{ess}(V)=\Sigma_{ac}(V) (=: E)
\]
and the closed set $E$ is also essentially closed, that is, $(x-r,x+r)\cap E$ is of positive
Lebesgue measure for every $x\in E$, $r>0$. With this latter assumption in place,
$\Sigma_{ac}=E$ will now imply that $\sigma_{ac}=E$ also.

Thus Theorem \ref{TBP} together with the obvious inclusion \eqref{1.1} imply that
\begin{equation}
\label{condW}
\Sigma_{ac}(W_{\pm}) = \sigma(W) = E, \quad \quad W\in \mathcal R (E)
\end{equation}
for every $W\in\omega(V)$. These are strong conditions and thus we can hope
to obtain rather detailed information on the possible potentials $W$, at least for nice sets $E$.

\begin{proof}[Proof of Theorem \ref{TDR}]
The following strategy suggests itself; see also \cite{GKT,GYud} and \cite[Chapter 8]{Teschl}
for very similar arguments in a similar context and especially \cite{Craig} for one of the
earliest uses of these ideas.

Given $W$ satisfying \eqref{condW}, consider the Green function of the whole line operator:
\[
G(z) = \langle \delta_0 , (J_W-z)^{-1} \delta_0 \rangle = \int_E \frac{d\mu(t)}{t-z} ,
\]
where $d\mu$ is the spectral measure of $J_W$ and the vector $\delta_0$.
It follows from this representation that $G(t)>0$ if $t<\alpha_0$, $G(t)<0$
if $t>\beta_{N+1}$, and $G(t)$ is (real and) strictly increasing in each gap
$(\beta_j,\alpha_j)$ ($j=1,\ldots ,N$).

We will work with the function $H(z)=-G^{-1}(z)$ instead, because, according to Theorem \ref{Tinv},
this is the function that has the decomposition $cH=M_++M_-$.

Of course, $H\in\mathcal H$, and the properties of $G$ observed above translate into
corresponding properties of $H$.
Since $W\in\mathcal R(E)$, hence $M_{\pm}\in\mathcal Q(E)$, we can also deduce that
$H\in\mathcal N(E)$. We again work with the exponential Herglotz representation of $H$ and introduce
\[
\xi(t) = \frac{1}{\pi} \lim_{y\to 0+} \textrm{Im}\left( \ln H(t+iy) \right) ,
\]
as in \eqref{5.5}. What we have learned above about $G$ and $H$ now says that
$\xi(t)=1$ if $t<\alpha_0$, $\xi(t)=0$ if $t>\beta_{N+1}$, $\xi(t)=1/2$ if
$t\in E$, and for each $j=1,2, \ldots , N$, there exists $\mu_j\in [\beta_j,\alpha_j]$ so that
\[
\xi(t) = \begin{cases} 0 & \beta_j<t<\mu_j \\ 1 & \mu_j<t<\alpha_j
\end{cases} .
\]
In other words, $\mu_j$ is defined as the unique point in the $j$th gap $(\beta_j,\alpha_j)$
for which $G(\mu_j)=0$, should there be such a point. If that is not the case, we let $\mu_j=\beta_j$
or $\mu_j=\alpha_j$, depending on which sign the values of $G$ on $(\beta_j,\alpha_j)$ have.

So, given the $\mu_j$'s, we have complete information about $\xi(t)$, and thus we can recover $H$,
up to a constant factor. If we pick this factor so that $H(z)=z+O(1)$ for large $|z|$, we obtain that
\begin{equation}
\label{formH}
H(z) = \sqrt{(z-\alpha_0)(z-\beta_{N+1})} \prod_{j=1}^N \frac{\sqrt{(z-\beta_j)(z-\alpha_j)}}{z-\mu_j} .
\end{equation}
Compare \cite[eq.\ (5.11)]{Craig}, \cite[Lemma 3.4]{GKT}, or \cite[Lemma 8.3]{Teschl}.

Let us now use Corollary \ref{C5.1} to find all potentials $W\in\mathcal R(E)$ that correspond to this $H$
and satisfy the additional condition that
\begin{equation}
\label{condW1}
\sigma(W)=E ;
\end{equation}
compare \eqref{condW}.
In other words, we must find all $F_+\in\mathcal Q(E)$ as in \eqref{5.7} which correspond to whole line
potentials $W$ with $\sigma(W)=E$. Clearly, $F_+$ is determined by $f$ from \eqref{5.7}, so it suffices
to discuss the possible choices for this function.

The measure $\rho$ associated with $H$ is purely absolutely continuous on $E$; this follows readily
from \eqref{formH}. Thus, on $E$, there is no choice: we must take $f=1/2$ by Proposition \ref{P5.1}.

This does not completely define $f$ almost everywhere with respect to $\rho$ because
$\rho(\{ \mu_j\} )>0$ for every $j$ for which $\mu_j\in (\beta_j, \alpha_j)$. This, too, follows
directly from \eqref{formH}. I now claim that only the choices $f(\mu_j)=0$ and $f(\mu_j)=1$ are consistent
with \eqref{condW1}. Indeed, if we had $0<f(\mu_j)<1$, then, since $\rho=\rho_++\rho_-$, both
half line problems would have an eigenvalue at $\mu_j$. This is equivalent to $f_{\pm}(0,\mu_j)=0$,
where $f_{\pm}$ are the solutions of \eqref{je} with $z=\mu_j$ that are square summable near
$\pm\infty$. It would then follow that $f_+$ and $f_-$ are actually multiples of one another,
that is, there exists a solution $f(\cdot, \mu_j)\in\ell_2(\Z)$ and hence $\mu_j\in\sigma(W)$.
This contradicts \eqref{condW1}.
Note also that no such problem occurs if $f(\mu_j)=0$ or $1$ because then $f_-(0,\mu_j)=0$,
$f_+(0,\mu_j)\not= 0$ or conversely, so there is no solution which is in $\ell_2(\Z)$.

We now summarize our discussion so far and add the (rather obvious) converse statement:
Given a set $E$ as in the statement of Theorem \ref{TDR}, the potentials $W$ that
satisfy \eqref{condW} are in natural one-to-one correspondence to the $N$-tuples
$(\widehat{\mu}_1, \ldots, \widehat{\mu}_N)$, where $\widehat{\mu}_j=(\mu_j, s_j)$ with
$\mu_j\in [\beta_j, \alpha_j]$ and $s_j=0, 1$, and if $\mu_j=\alpha_j$ or $\mu_j=\beta_j$,
then $(\mu_j,0)=(\mu_j,1)$ are identified.

More precisely, given such an $N$-tuple $(\widehat{\mu}_1, \ldots, \widehat{\mu}_N)$,
we define $H\in\mathcal N(E)$ by \eqref{formH} and $F_+$ as in \eqref{5.7}, with
$f=1/2$ on $E$ and $f(\mu_j)=s_j$. In this context, also
recall that $\rho(\{ \mu_j\} )=0$ if $\mu_j=\beta_j$ or $\mu_j=\alpha_j$, so in this case,
it wouldn't have been necessary to specify $f(\mu_j)$, and thus $s_j$ becomes irrelevant.

We have now defined $f$ almost everywhere with respect to
$\rho$ in such a way that the pair
$(H,F_+)$ satisfies the conditions from Corollary \ref{C5.1}.

From the construction, it is clear that the potentials $W$ obtained in this way
satisfy $W\in\mathcal R(E)$,
$\sigma_{ess}(W)=E$. Moreover, a point $t\notin E$ can only be an eigenvalue if
either $M_+(t)=-M_-(t)$ or both $M_+$ and $M_-$ have a pole at $t$.
Indeed, this is the condition for the two half line $\ell_2$ solutions $f_{\pm}(\cdot, t)$
to match at $n=0$. However, the second condition leads us back to the $\mu_j$, and we have been careful
to make sure that these are not eigenvalues. The first condition would imply that $H(t)=0$, but
\eqref{formH} shows that this does not happen outside $E$. We conclude that
the construction just described does produce a $W$ satisfying \eqref{condW}.

Conversely, our discussion
above has shown that if $W$ satisfies \eqref{condW}, then it arises in this way.

The set of parameters $\widehat{\mu}=(\widehat{\mu}_1,\ldots , \widehat{\mu}_N)$ can
be naturally identified with an $N$-dimensional torus $S^1\times \cdots \times S^1$. Furthermore, it
is quite clear from the construction that then the map
$\widehat{\mu}\mapsto M_+\in\mathcal H$ becomes a continuous
injective map if we, as usual, endow $\mathcal H$ with the topology of uniform convergence on
compact subsets of $\C^+$. We only need to verify that $\rho(\{ \mu_j \} )\to 0$ as
$\mu_j\to\beta_j$ or $\alpha_j$, in order to rule out problems at those points $\widehat{\mu}$
for which $\mu_j=\beta_j$ or $\alpha_j$ for some $j$.
However, this claim follows immediately from \eqref{formH}
if we make use of the general formula
\[
\rho(\{ x \} ) = -i\lim_{y\to 0+} yH(x + iy) .
\]

If we now define $\mathcal T_N(E)$ as the sets of potentials $W$
satisfying \eqref{condW}, then, as just explained, we have a bijection
\[
S^1 \times \cdots \times S^1 \to \mathcal T_N(E) , \quad\quad \widehat{\mu} \mapsto W .
\]
Since the map $M_+\mapsto W$ is continuous by Lemma \ref{L4.1} and Proposition \ref{P4.1}(e),
this correspondence is continuous, too, and $\mathcal T_N(E)$ is homeomorphic to an $N$-dimensional
torus, as claimed.

This proves Theorem \ref{TDR}, except for the additional claim that $\mathcal T_N(E)$ coincides
with the set of finite gap potentials. This just follows from the fact that the $m$ functions
$M_+$ that we obtain here exactly coincide with those of the finite gap potentials with spectrum $E$, and in
both cases, by the reflectionless condition, these determine the potential uniquely.
See \cite[Sect.~8.3]{Teschl}, especially formula (8.87) and Theorem 8.17.
\end{proof}
This contains Corollary \ref{C1.2} as the special case $N=0$. However, it is also instructive
to run through the above argument again to explicitly confirm the claim made in \eqref{1.8}.
So assume now that
\[
\sigma(W)=[-2,2], \quad\quad W\in\mathcal R(-2,2) .
\]
There is no gap and thus \eqref{formH} simplifies to
\[
H(z) = \sqrt{(z+2)(z-2)} .
\]
The only $F_+\in\mathcal Q(-2,2)$ compatible with this
$H$ and satisfying the assumptions of Corollary \ref{C5.1}
is given by
\[
F_+(z) = \frac{1}{2\pi} \int_{-2}^2 \frac{\sqrt{4-t^2}}{t-z}\, dt .
\]
Since $\chi_{(-2,2)}(t)\sqrt{4-t^2}\, dt/(2\pi)$ is the spectral measure
of the free Jacobi matrix $a=1$, $b=0$ on $\Z_+$, we now obtain \eqref{1.8}.

Although this is somewhat off topic here, let me also briefly describe how one can continue from here
if a deeper analysis of the (finite gap) potentials $W\in\mathcal T_N(E)$ is desired. The standard theory
proceeds as follows: Instead of cutting the whole line into two half lines at $n=0$, we can of course
also cut at an arbitrary $n\in\Z$. In this way, we obtain a sequence of
$n$-dependent parameters $\widehat{\mu}(n)$
for every fixed $W\in\mathcal T_N(E)$. The crucial fact is this: by conjugating with the Abel-Jacobi map
of the Riemann surface of $w^2=\prod (z-\alpha_j)(z-\beta_{j+1})$, the map evolving the $\widehat{\mu}(n)$
becomes translation on another $N$-dimensional torus (the real part of the Jacobi variety). This
proves that $W$ is quasi-periodic and gives a rather explicit description. See the references mentioned
above for more detailed information \cite{BGHT,Dub,KrRem,McK,Mum,Teschl,vMoer,vMoerMum}.
\section{Two counterexamples}
As our first illustrative example relevant to Theorem \ref{TBP}, we just recall the properties
of a model investigated in depth by Stolz in \cite{Sto}. In fact, Stolz discusses a general class of slowly
oscillating potentials, but we will only consider the potential
\[
V(n) = \cos \sqrt{n} .
\]
Then \cite[Theorem 1]{Sto}, applied to the case at hand, says that
\[
\sigma_{ess}(V)=[-3,3], \quad \Sigma_{ac}(V)=[-1,1]
\]
(actually, the latter statement is not given in literally this form, but it is easily
extracted from \cite{Sto}). On the other hand, since $V(n)$ is almost constant on long intervals
for all large $n$, it is clear that
\[
\omega(V) = \left\{ W^{(a)}(n) \equiv a : -1\le a \le 1 \right\} .
\]
These limit potentials potentials satisfy
\[
\sigma(W^{(a)})=\Sigma_{ac}(W_{\pm}^{(a)}) = [-2+a, 2+a], \quad\quad W^{(a)}\in\mathcal R(-2+a,2+a) .
\]
Note that for all $a\in [-1,1]$, we have that
\[
\Sigma_{ac}(V)\subset \Sigma_{ac}(W_{\pm}^{(a)}), \quad\quad \sigma(W^{(a)})\subset \sigma_{ess}(V) ,
\]
as asserted by Corollary \ref{C1.1} and \eqref{1.1}. Both inclusions are strict.

The potentials $W^{(a)}$ have two additional properties, which are not guaranteed by
Theorem \ref{TBP}: They are reflectionless on
the larger (than $\Sigma_{ac}(V)$) sets $\Sigma_{ac}(W_{\pm}^{(a)})=[-2+a,2+a]$, and we can obtain
$\Sigma_{ac}(V)$ as
\begin{equation}
\label{7.1}
\Sigma_{ac}(V) = \bigcap_{W\in\omega(V)} \Sigma_{ac}(W_{\pm}) .
\end{equation}
Neither of these properties holds in general. In fact, it could also be argued that \eqref{7.1}
doesn't make much sense in general because $\Sigma_{ac}$ is only determined up to sets of measure
zero. This objection, however, is somewhat beside the point because the following example will
reveal more serious problems.

The basic idea behind this second example
is to use inverse scattering theory to come up with a suitable
(whole line) potential $W^{(0)}=(a_0,b_0)$ that will be the fundamental building block.
I should also point out that Molchanov \cite{Mol} has analyzed similar examples in great detail,
using related ideas.

We will not enter a serious discussion of inverse scattering theory here. Rather, we will
just extract what we need and
refer the reader to \cite{GerC,VolYud} and especially \cite[Chapter 10]{Teschl}
for a thorough treatment. However, I will mention one (well known) basic fact in
Proposition \ref{P7.1} below, in order to motivate and illuminate the construction.

Inverse scattering theory yields
the existence of (whole line) Jacobi coefficients $W^{(0)}=(a_0,b_0)$ with
the following set of properties: $\Sigma_{ac}(W_{\pm}^{(0)})=[-2,2]$,
\[
a_0(n) \to 1, \quad b_0(n)\to 0 \quad\quad( |n|\to\infty ) ,
\]
and for $\varphi\in (0,\pi/2)$, the Jacobi equation
\begin{equation}
\label{je3}
a_0(n)f(n+1)+a_0(n-1)f(n-1)+b_0(n)f(n)=2\cos\varphi\, f(n)
\end{equation}
has a solution
$f(n,\varphi)$ satisfying the asymptotic formulae
\begin{equation}
\label{jostsln}
f(n,\varphi) = \begin{cases} e^{in\varphi} + o(1) & n\to -\infty \\
e^{i\psi}e^{in\varphi} + o(1) & n\to\infty  \end{cases} .
\end{equation}
The angle $\psi$ will usually depend on $\varphi$. Note also that
the spectral parameter $t=2\cos\varphi$ varies over $(0,2)$ if $\varphi\in (0,\pi/2)$.
The significance of \eqref{jostsln} will be discussed further in a moment, but we can
immediately make the clarifying remark that \eqref{jostsln}
will imply that $W^{(0)}\in\mathcal R(0,2)$.

Finally, we also demand that $W^{(0)}$ is \textit{not }reflectionless anywhere on $(-2,0)$.
More precisely, if
$A\subset (-2,0)$ has positive measure, then $W^{(0)}\notin\mathcal R(A)$.

It may seem that we are asking for a lot here,
but actually the existence of potentials $W^{(0)}$ with these properties is a rather easy consequence of
inverse scattering theory and we don't need anything close to the full force of this machinery here.
Just start out with a smooth (and let's say: real valued and even) reflection coefficient that is zero
precisely on $|\varphi|\le\pi/2$. The basic result that will do all the work here
is \cite[Theorem 10.12]{Teschl}. We actually obtain more precise information on $W^{(0)}$
from this, but what we have stated above will suffice for our purposes here.

We don't want to enter a discussion of the technical details of inverse scattering theory,
so I will leave the matter at that. Let me just mention one illuminating fact. We need some notation.
If $a(n)-1$, $b(n)$ decay sufficiently rapidly ($a-1,b\in\ell_1(\Z)$ will suffice for the few simple
remarks I want to make here, but in inverse scattering theory, one typically needs stronger assumptions),
then the Jacobi equation \eqref{je3} with $\varphi\notin\Z\pi$ has \textit{Jost solutions,}
that is, solutions of the asymptotic form $f_{\pm} = e^{\pm in\varphi}+o(1)$ as $n\to\infty$.
These are linearly independent and thus form a basis of the solution space.
Moreover, there are of course other solutions satisfying similar formulae near $-\infty$.
In particular, we can take such a solution and expand it in terms of $f_{\pm}$. In other words, there
exist $T(\varphi),R(\varphi)\in\C$ so that the following formula describes the asymptotics
of a certain solution $f$:
\[
f(n,\varphi) = \begin{cases} e^{in\varphi} + o(1) & n\to -\infty \\
T(\varphi)^{-1}(e^{in\varphi}+R(\varphi)e^{-in\varphi}) + o(1) & n\to\infty  \end{cases} .
\]
The coefficients $T$, $R$ defined in this way are called the
\textit{transmission} and \textit{reflection coefficients,}
respectively. Constancy of the Wronskian implies that $|T|^2+|R|^2=1$, so we now see that \eqref{jostsln}
corresponds to the special case where $R\equiv 0$ on $\varphi\in(0,\pi/2)$.
\begin{Proposition}
\label{P7.1}
Suppose that $a-1,b\in\ell_1(\Z)$. Let
\[
A = \left\{ 2\cos\varphi : R(\varphi)=0 \right\} .
\]
Then $(a,b)$ is reflectionless precisely on $A$. In other words,
$(a,b)\in\mathcal R(A)$, but $(a,b)\notin\mathcal R(B)$ if $|B\setminus A|>0$.
\end{Proposition}
This follows quickly from the observation that the Jost solutions $f_{\pm}$ are
the boundary values (as $z\to t\in\R$) of the $\ell_2$ solutions $f_{\pm}(\cdot, z)$
defined earlier (see Sect.~3). We will not provide any details here.

The Proposition finally explains the terminology: at least in a scattering situation,
a potential is reflectionless
precisely on the set on which the reflection coefficient vanishes.
It is now also clear that a potential $W^{(0)}$ with the properties given above might be relevant for
the issues we are interested in here. Let us now take a closer look at such an example. Basically,
we will just assemble $V$ from repeated copies of $W^{(0)}$, but since $W^{(0)}$ is not compactly supported,
we will also need cut-offs.
\begin{Theorem}
\label{T7.1}
Let
\[
V(n)=(a(n),b(n)) = \begin{cases}  W^{(0)}(n-c_j) & c_j-L_j\le n \le c_j+L_j \\
(1,0) & \textrm{\rm otherwise} \end{cases} .
\]
If $L_j$ increases sufficiently rapidly and the $c_j$ also increase so fast that the
intervals $\{ c_j-L_j, \ldots , c_j + L_j \}$ are disjoint, then the (half line) Jacobi matrix with
coefficients $V$ satisfies
\[
\Sigma_{ac}(V) = [0,2] .
\]
On the other hand,
\[
\omega(V) = \left\{ S^nW^{(0)} : n\in\Z \right\} \cup \left\{ W\equiv (1,0) \right\} ,
\]
so, in particular,
\[
\Sigma_{ac}(W_{\pm}) = [-2,2] \quad\text{for all }\: W\in\omega(V) .
\]
Moreover, if $W\in\omega(V)$, $W\not= (1,0)$, then $W\notin\mathcal R(\Sigma_{ac}(W_{\pm}))$.
\end{Theorem}
This shows, first of all, that one needn't bother with trying to make sense out of
\eqref{7.1} in general since it is wrong anyway. The final statement seems more
important still: While we of course always have that $W\in\mathcal R(\Sigma_{ac}(V))$,
as asserted by Theorem \ref{TBP}, it is \textit{not }true in general that $W\in\omega(V)$
is also reflectionless on the possibly larger sets $\Sigma_{ac}(W_{\pm})$.
\begin{proof}
Only the inclusion
\begin{equation}
\label{7.2}
\Sigma_{ac}(V)\supset [0,2]
\end{equation}
needs serious proof; everything
else then falls into place very quickly. Indeed, to determine $\omega(V)$, it suffices
to recall that $W^{(0)}=(a_0,b_0)\to (1,0)$ as $|n|\to\infty$. Since $W^{(0)}$ is not
reflectionless anywhere on the complement of $[0,2]$, it follows from Theorem \ref{TBP} that
$\Sigma_{ac}(V)\subset [0,2]$.

So let us now prove \eqref{7.2}. The crucial observation is the following: the condition
of the reflection coefficient being zero forces the transfer matrix to be close to a rotation
asymptotically. Let us make this more precise. We again assume that $\varphi\notin\Z\pi$.
Given a solution $y$ of \eqref{je3}, we then introduce
\[
Y(n) = \begin{pmatrix} \sin\varphi & 0 \\ -\cos\varphi & 1 \end{pmatrix}
\begin{pmatrix} y(n-1) \\ y(n) \end{pmatrix} .
\]
This may look somewhat arbitrary at first sight but is actually a natural thing to do
because length and direction of $Y$ are the familiar Pr\"ufer variables. Compare \cite{KLS,KruRem}.
Let $T(L,\varphi)\in\R^{2\times 2}$ be the matrix that moves $Y(n)$ from $n=-L$ to $n=L$, that is,
$TY(-L)=Y(L)$.

To illustrate the basic mechanism in a situation that is as simple as possible,
suppose for a moment that we had a solution of the type \eqref{jostsln}, but with no errors
(so, formally, $o(1)=0$ in \eqref{jostsln}). This is a fictitious situation because the
other properties of $W^{(0)}$ would then become contradictory,
but let us not worry about this now; it will be easy to incorporate the error terms afterwards.
The transfer matrix $T(L,\varphi)$ then would have to satisfy
\[
e^{-i(L+1)\varphi} T(L,\varphi) \begin{pmatrix} \sin\varphi & 0 \\ -\cos\varphi & 1 \end{pmatrix}
\begin{pmatrix} 1 \\ e^{i\varphi} \end{pmatrix} =
e^{i\psi} e^{i(L-1)\varphi} \begin{pmatrix} \sin\varphi & 0 \\ -\cos\varphi & 1 \end{pmatrix}
\begin{pmatrix} 1 \\ e^{i\varphi} \end{pmatrix} ,
\]
or, equivalently,
\begin{equation}
\label{7.5}
T(L,\varphi) \begin{pmatrix} 1 \\ i \end{pmatrix} = e^{i\psi} e^{2iL\varphi}
\begin{pmatrix} 1 \\ i \end{pmatrix} .
\end{equation}
Recall that the entries of $T$ are real. Thus \eqref{7.5} implies that
\[
T(L,\varphi) = \begin{pmatrix} \cos\theta & \sin\theta \\ -\sin\theta & \cos\theta \end{pmatrix} ,
\quad\quad \theta \equiv \psi + 2L\varphi .
\]
An analogous calculation still works if we keep the error terms $o(1)$;
it now follows that if \eqref{jostsln} holds, then
\begin{equation}
\label{7.8}
T(L,\varphi) = \begin{pmatrix} \cos\theta & \sin\theta \\ -\sin\theta & \cos\theta \end{pmatrix}
+ S(L,\varphi), \quad\quad
\lim_{L\to\infty}\|S(L,\varphi)\| = 0 .
\end{equation}

Thus, by dominated convergence, we can now pick $L_n\in\N$ so large that
\[
\left| \left\{ \varphi\in(0,\pi/2): \|S(L_n,\varphi)\|\ge 2^{-n} \right\} \right| < 2^{-n} .
\]
We may further demand that
\begin{equation}
\label{7.9}
|a_0(k)-1| , \: |b_0(k)| < 2^{-n} \quad\quad \textrm{ for } |k| \ge L_n .
\end{equation}
The Borel-Cantelli Lemma guarantees that for almost every $\varphi\in (0,\pi/2)$, we will
eventually have that
\[
\|S(L_n,\varphi)\| < 2^{-n} \quad\quad (n\ge n_0=n_0(\varphi)) .
\]
Now define $V$ as described in the statement of the theorem, let $Y$ correspond to
a solution of the Jacobi equation with these coefficients, and put
$R_n(\varphi)=\|Y(c_n+L_n,\varphi)\|$. Recall also that the transfer matrix over
an interval is \textit{exactly }a rotation if $V=(1,0)$ on that interval.
Finally, \eqref{7.9} makes sure that the one-step transfer matrices at the gluing points
$c_n\pm L_n$ also differ from a rotation only by a correction of order $O(2^{-n})$.

Thus, putting things together, we find
that for almost every $\varphi\in(0,\pi/2)$, we
have estimates of the form
\[
R_n(\varphi) \le \left( 1+C_{\varphi} 2^{-n} \right) R_{n-1}(\varphi)
\]
for all $n\ge n_0(\varphi)$. It follows that
\[
\limsup_{n\to\infty} R_n(\varphi) < \infty
\]
for almost every $\varphi\in (0,\pi/2)$, and this implies \eqref{7.2} by
\cite[Proposition 2.1]{Remac}.
\end{proof}
With a more careful analysis, one can also establish that \eqref{7.8} holds
uniformly on $\varphi\in [\epsilon, \pi/2-\epsilon]$ and this lets one show
that the spectrum is actually \textit{purely }absolutely continuous on $(0,2)$,
by using a criterion like \cite[Theorem 1.3]{LS}. However, this improvement does
not seem to be of particular interest here, so we have taken an armchair approach instead
and given preference to the technically lighter treatment.
\begin{appendix}
\section{Proof of Theorem \ref{TBPorig}}
The two key notions in this proof are (pseudo)hyperbolic distance and harmonic measure on $\C^+$.
We quickly summarize the basic facts that are needed in the sequel. See \cite{BP1,BP2} for
a more detailed treatment and also \cite{Krantz,Siegel} for background information.

As in \cite{BP1,BP2}, we define the pseudohyperbolic distance of two points $w,z\in\C^+$ as
\begin{equation}
\label{gamma}
\gamma(w,z) = \frac{|w-z|}{\sqrt{\textrm{Im }w}\sqrt{\textrm{Im }z}} .
\end{equation}
This is perhaps not the most commonly used formula and it doesn't satisfy the
triangle inequality, but it is well adapted to our needs here. See
\cite[Proposition 1]{BP2} for the relation of $\gamma$ to hyperbolic distance.
Note that \cite[Proposition 1]{BP2} in particular says that $\gamma$ is an increasing function of
hyperbolic distance. As a consequence,
holomorphic maps $F:\C^+\to \C^+$ are distance decreasing: $\gamma(F(w),F(z))\le \gamma(w,z)$
if $F\in\mathcal H$. In particular, for automorphisms $F\in\textrm{Aut}(\C^+)$, we have equality here.
Recall also
that these are precisely the linear fractional transformations
\[
z \mapsto \frac{az+b}{cz+d}
\]
with $a,b,c,d\in\R$, $ad-bc>0$. It will be convenient to use matrix notation for (general)
linear fractional transformations. That is, if $a,b,c,d\in\C$, $ad-bc\not=0$, and $z\in\C$,
we will define
\begin{equation}
\label{lft}
S z = \frac{az+b}{cz+d},\quad\quad
S\equiv \begin{pmatrix} a & b \\ c & d \end{pmatrix} .
\end{equation}
This is best thought of as the matrix $S$ acting in the usual way on the vector
$(z,1)^t=[z:1]$ of the homogeneous coordinates of $z=[z:1]\in\C\subset\C\bbP^1$. The image
vector $S(z,1)^t$ records the homogeneous coordinates of the image point under the linear
fractional transformation. In particular, \eqref{lft} then describes
the action of $S$ on the Riemann sphere $\C_{\infty}\cong\C\bbP^1$.

Hyperbolic distance and harmonic measure are intimately related: If $w,z\in\C^+$ and
$S\subset\R$ is an arbitrary Borel set, then
\begin{equation}
\label{A.1}
\left| \omega_w(S) - \omega_z(S) \right| \le \gamma(w,z)
\end{equation}
(this will suffice for our purposes here, but
see also \cite[Proposition 2]{BP2} for an interesting stronger statement).
To prove \eqref{A.1}, fix $S\subset\R$ with $|S|>0$,
and recall that $z\mapsto \omega_z(S)$ is a positive harmonic
function on $\C^+$. This function has a harmonic conjugate $\alpha(z)$; in other words,
$F(z)=\alpha(z)+i\omega_z(S)\in\mathcal H$, and by the distance decreasing property of such functions,
we obtain that
\[
\gamma(w,z)\ge \gamma(F(w),F(z)) \ge
\frac{\left| \omega_w(S) - \omega_z(S) \right| }{\sqrt{\omega_w(S)}\sqrt{\omega_z(S)}}
\ge \left| \omega_w(S) - \omega_z(S) \right| .
\]
\begin{Lemma}
\label{LA.1}
Let $A\subset\R$ be a Borel set with $|A|<\infty$. Then
\[
\lim_{y\to 0+} \sup_{F\in\mathcal H; S\subset\R} \left|
\int_A \omega_{F(t+iy)}(S)\, dt - \int_A \omega_{F(t)}(S)\, dt \right| = 0 .
\]
\end{Lemma}
This stunning result is Theorem 1 of \cite{BP2}.

The point here is the \textit{uniform} convergence. For fixed $F$, the statement
follows immediately from Proposition \ref{P2.2}(c) and the obvious fact that $F(z+iy)\to F(z)$
locally uniformly as $y\to 0$.
\begin{proof}
This will follow from the neat identity
\begin{equation}
\label{A.2}
\omega_{F(z)}(S) = \int_{-\infty}^{\infty} \omega_{F(t)}(S)\, d\omega_z(t) ,
\end{equation}
which is valid for all Borel sets $S\subset\R$ and $z\in\C^+$. To prove \eqref{A.2},
it suffices to observe that both sides are bounded,
non-negative harmonic functions of $z\in\C^+$ with the same
boundary values $\omega_{F(t)}(S)$ for almost every $t\in\R$. Therefore, they must be identical.

Fubini's Theorem now shows that
\begin{align*}
\int_A \omega_{F(t+iy)}(S)\, dt & =
\int_A dt \int_{-\infty}^{\infty} d\omega_{t+iy}(u)\, \omega_{F(u)}(S) \\
& = \frac{1}{\pi} \int_A dt \int_{-\infty}^{\infty} du\, \frac{y}{(u-t)^2+y^2}\, \omega_{F(u)}(S) \\
& = \int_{-\infty}^{\infty} \omega_{u+iy}(A)\, \omega_{F(u)}(S)\,du .
\end{align*}
Therefore,
\begin{align*}
\left| \int_A \omega_{F(t+iy)}(S)\, dt - \int_A \omega_{F(t)}(S)\, dt \right| & =
\left| \int_{-\infty}^{\infty} \omega_{F(t)}(S)
\left( \omega_{t+iy}(A) - \chi_A(t) \right) \, dt \right| \\
& \le \max_{B=A, A^c} \int_B \omega_{t+iy}(B^c)\, dt = \int_A \omega_{t+iy}(A^c) \, dt .
\end{align*}
The inequality follows because $0\le \omega_{F(t)}(S)\le 1$ and the
second factor satisfies $\omega_{t+iy}(A) - \chi_A(t) \ge 0$ for $t\in A^c$ and it is $\le 0$
if $t\in A$, so by integrating over just one of these sets we only avoid cancellations.
If we then use the definition of $\omega_z$ and Fubini's Theorem, we see that the two integrals
from the maximum are equal to one another.

We have now estimated the difference from the statement of Lemma \ref{LA.1} by
\[
\epsilon_A(y):=\int_A \omega_{t+iy}(A^c) \, dt ,
\]
a quantity that is independent of both $F$ and $S$. To show that $\epsilon_A(y)\to 0$ as $y\to 0+$,
recall that by Lebesgue's differentiation theorem, we have that $|A^c\cap (t-h,t+h)|=o(h)$ for almost
all $t\in A$. For such a $t$, we obtain that
\begin{align*}
\omega_{t+iy}(A^c) & \le \frac{1}{\pi}\int_{A^c\cap(t-Ny,t+Ny)} \frac{y}{(s-t)^2+y^2}\, ds +
\frac{1}{\pi}\int_{|s-t|\ge Ny} \frac{y}{(s-t)^2+y^2}\, ds \\
& = No(1) + 1 - \frac{2}{\pi}\arctan N
\end{align*}
as $y\to 0+$. By taking $y$ small enough and noting that $N>0$ is arbitrary, we see that
$\omega_{t+iy}(A^c)\to 0$ for almost all $t\in A$, and thus indeed
$\epsilon_A(y)\to 0$ by dominated convergence.
\end{proof}
We are interested in the asymptotic behavior of the $m$ functions $m_{\pm}(n,z)$, as $n\to\infty$.
We recall the definitions. For $z\in\C^+$, let $f_{\pm}(n,z)$ be the solutions of
\begin{equation}
\label{je2}
a(n)f(n+1)+a(n-1)f(n-1)+b(n)f(n)=zf(n)
\end{equation}
satisfying $a(0)f_-(0,z)=0$ and $f_+(\cdot,z)\in\ell_2(\Z_+)$, respectively. These are unique
up to constant factors. Then
\[
m_{\pm}(n,z)=\mp \frac{f_{\pm}(n+1,z)}{a(n)f_{\pm}(n,z)} ,
\]
From \eqref{je2}, we can easily extract the matrices $T_{\pm}$ that describe the evolution of the
vectors $(f(n+1), \mp a(n)f(n))^t$. Moreover, the components of these vectors are homogeneous
coordinates of the numbers $m_{\pm}(n,z)$. We thus find that
\begin{equation}
\label{ric}
m_{\pm}(n,z) = T_{\pm}(a(n),b(n),z) m_{\pm}(n-1,z) ,\quad\quad
T_{\pm}(a,b,z) \equiv \begin{pmatrix} \frac{z-b}{a} & \pm \frac{1}{a} \\ \mp a & 0 \end{pmatrix} .
\end{equation}
Here, we use the matrix notation for linear fractional transformations, as introduced in \eqref{lft}.
Of course, if written out, \eqref{ric} gives us the familiar Riccati equations for $m_{\pm}$
(see, for example, \cite[Eqns (2.11), (2.13)]{Teschl}).

We will use the abbreviations $m_+(z)\equiv m_+(0,z)$ and $T_{\pm}(n,z)\equiv T_{\pm}(a(n),b(n),z)$,
and we also introduce
\[
P_{\pm}(n,z) := T_{\pm}(n,z)T_{\pm}(n-1,z) \cdots T_{\pm}(1,z) .
\]
By iterating \eqref{ric} (and noting that $m_-(0,z)=\infty$), we then obtain that
\begin{equation}
\label{A.3}
m_+(n,z) = P_+(n,z)m_+(z), \quad\quad m_-(n,z) = P_-(n,z)\infty .
\end{equation}

We also observe the following properties of the linear fractional transformations $T_{\pm}$:
First of all, if $z\in\R$ (and $a>0$, $b\in\R$),
then $T_{\pm}(a,b,z)\in\textrm{Aut}(\C^+)$, the automorphisms of $\C^+$.
If $z\in\C^+$, then $T_-(a,b,z)\in\mathcal H$, while $T_+$ does \textit{not} map $\C^+$ to
itself then.

Let us now return to the proof of Theorem \ref{TBPorig}. Suppose we did not know that
$m_-(0,z)=\infty$, but only that $m_-(0,z)\in\overline{\C^+}$.
The above remarks together with \eqref{A.3} make
it clear that then the hyperbolic diameter of the
set of possible values of $m_-(n,z)$ decreases as $n\to\infty$.
The following Lemma and especially Corollary \ref{CA.1} make this more precise.
\begin{Lemma}
\label{LA.2}
Let $a>0$, $b\in\R$, $z\in\C$ with $y\equiv\text{\rm Im }z\ge 0$. Suppose that
\[
w_j = T_-(a_0,b_0,z)\zeta_j, \quad\quad \zeta_j\in\C^+, \quad a_0>0 .
\]
Then
\[
\gamma\left( T_-(a,b,z)w_1, T_-(a,b,z)w_2 \right) \le \frac{1}{1+(y/a_0)^2} \,
\gamma(w_1,w_2) .
\]
\end{Lemma}
\begin{Corollary}
\label{CA.1}
Suppose that $a(n)\le A$, and let $K$ be a compact subset of $\C^+$. Then
\[
\lim_{n\to\infty} \gamma(m_-(n,z),P_-(n,z)w) = 0,
\]
uniformly in $z\in K$, $w\in\overline{\C^+}$.
In fact, $\gamma\le Cq^n$ for $n\ge 2$, where
we may take $q=1/(1+(\delta/A)^2)<1$ if $\textrm{\rm Im }z\ge\delta >0$ for all $z\in K$.
\end{Corollary}
\begin{proof}
The Corollary is immediate from the Lemma if we also note that by Weyl theory (or inspection), the set
$\{ P_-(n,z)w: w\in\overline{\C^+}, z\in K \}$ is a compact subset of $\C^+$ for $n\ge 2$.

So it suffices to prove the Lemma. Now
\[
T\equiv T_-(a,b,z) = \begin{pmatrix} 1/a & -b/a \\ 0 & a \end{pmatrix}
\begin{pmatrix} 1 & z \\ 0 & 1 \end{pmatrix}
\begin{pmatrix} 0 & -1 \\ 1 & 0 \end{pmatrix} \equiv AS(z)J ,
\]
and $A,J\in\textrm{Aut}(\C^+)$, that is, $A$ and $J$ are isometries with respect to $\gamma$. Therefore, we
can put $u_j=Jw_j$ and we then have that $\gamma(Tw_1,Tw_2)=\gamma(Su_1,Su_2)$ and
$\gamma(w_1,w_2)=\gamma(u_1,u_2)$. Moreover, from the definition \eqref{gamma} it is immediate that
\begin{equation}
\label{A.4}
\frac{\gamma(S(z)u_1, S(z)u_2)}{\gamma(u_1,u_2)} = \frac{\gamma(u_1+z,u_2+z)}{\gamma(u_1,u_2)}
= \left( \frac{\textrm{Im }u_1}{y+\textrm{Im }u_1}\right)^{1/2}
\left( \frac{\textrm{Im }u_2}{y+\textrm{Im }u_2}\right)^{1/2} .
\end{equation}
The hypothesis on $w_j$ says that
\[
w_j = \frac{z-b_0}{a_0^2} + \frac{Z_j}{a_0^2} , \quad\quad Z_j=-1/\zeta_j \in\C^+ .
\]
Thus $\textrm{Im }w_j\ge y/a_0^2$, and since $u_j=-1/w_j$, it follows $\textrm{Im }u_j\le a_0^2/y$.
Therefore, the asserted estimate follows from \eqref{A.4}
\end{proof}
We now have all the tools to finish the
\begin{proof}[Proof of Theorem \ref{TBPorig}]
Let $A\subset\Sigma_{ac}$,
$|A|<\infty$, and let $\epsilon>0$ be given. As the first step of the proof, decompose $A=A_0\cup A_1 \cup
\ldots \cup A_N$. We pick these sets in such a way that
$m_+(t)\equiv\lim_{y\to 0+}m_+(t+iy)$ exists and $m_+(t)\in\C^+$
on $\bigcup_{j=1}^N A_j$. Moreover, we demand that there exist $m_j\in\C^+$ so that
\begin{equation}
\label{A.12}
\gamma(m_+(t),m_j)< \epsilon \quad\quad (t\in A_j , j=1, \ldots, N)
\end{equation}
and $|A_0|<\epsilon$. Finally, we require that $A_j$ is bounded for $j\ge 1$.

To find $A_j$'s with these properties, first of all put all $t\in A$ for which
$m_+(t)$ does not exist or does not lie in $\C^+$ into $A_0$
(so far, $|A_0|=0$). Then pick (sufficiently large)
compact subsets $K\subset \C^+$, $K'\subset\R$ so that
$A_0=\{ t\in A: m_+(t)\notin K \textrm{ or }t\notin K' \}$ satisfies $|A_0|<\epsilon$.
Subdivide $K$ into finitely many subsets of hyperbolic diameter less than $\epsilon$, then take
the inverse images under $m_+$ of these subsets, and finally intersect with $K'$ to obtain
the $A_j$ for $j\ge 1$.

It is then also true that $m_+(n,t)$ exists and lies
in $\C^+$ for arbitrary $n\in\Z_+$ if $t\in \bigcup_{j=1}^N A_j$.
Moreover, since $P_+(n,t)\in\textrm{Aut}(\C^+)$, we obtain from \eqref{A.3} and \eqref{A.12} that also
\begin{equation}
\label{A.5}
\gamma(m_+(n,t),P_+(n,t)m_j) < \epsilon \quad\quad (t\in A_j, j=1,\ldots, N)
\end{equation}
We may now use \eqref{A.1} and integrate to see that for arbitrary Borel sets $S\subset\R$,
\begin{equation}
\label{A.11}
\left| \int_{A_j} \omega_{m_+(n,t)}(S)\, dt - \int_{A_j}
\omega_{P_+(n,t)m_j}(S)\, dt \right| \le \epsilon |A_j| .
\end{equation}
Use the definition of harmonic measure (see \eqref{harmmeas}) to rewrite the second integrand as
\begin{equation}
\label{A.18}
\omega_{P_+(n,t)m_j}(S) = \omega_{-\overline{P_+(n,t)m_j}}(-S) .
\end{equation}
Moreover, and this seems to be one of the most important steps of the whole proof, we can further
manipulate this as follows:
\begin{equation}
\label{A.6}
\omega_{-\overline{P_+(n,t)m_j}}(-S) = \omega_{P_-(n,t)(-\overline{m_j})}(-S)
\end{equation}
To see this, just note that the linear fractional transformation that is multiplication by $-1$
corresponds to the matrix $\bigl( \begin{smallmatrix} 1 & 0 \\ 0 & -1 \end{smallmatrix} \bigr)$ and
\[
\begin{pmatrix} 1 & 0 \\ 0 & -1 \end{pmatrix} T_+(a,b,z) \begin{pmatrix} 1 & 0 \\ 0 & -1 \end{pmatrix}
= T_-(a,b,z) .
\]
This implies that we also have that
\begin{equation}
\label{Abasic}
\begin{pmatrix} 1 & 0 \\ 0 & -1 \end{pmatrix} P_+(n,z)
= P_-(n,z) \begin{pmatrix} 1 & 0 \\ 0 & -1 \end{pmatrix} ,
\end{equation}
and this (for $z=t$) gives \eqref{A.6}.
(This is the key calculation that was alluded to in Sect.~1.6; see \eqref{1.11}.)

Use Lemma \ref{LA.1} to find a $y>0$ so that
\begin{equation}
\label{A.8}
\left| \int_{A_j} \omega_{F(t+iy)}(-S)\, dt - \int_{A_j} \omega_{F(t)}(-S)\, dt \right| \le \epsilon |A_j|
\end{equation}
for all $F\in\mathcal H$, all Borel sets $S\subset\R$ and $j=1,\ldots, N$. By Corollary \ref{CA.1},
we can now find an $n_0\in\N$ so that
\[
\gamma\left( m_-(n,t+iy),P_-(n,t+iy)(-\overline{m_j})\right) < \epsilon
\]
for all $n\ge n_0$, $t\in A_j$, $j=1,\ldots,N$. Use \eqref{A.1} and integrate over $A_j$. This gives
\[
\left| \int_{A_j} \omega_{P_-(n,t+iy)(-\overline{m_j})}(-S)\, dt -
\int_{A_j} \omega_{m_-(n,t+iy)}(-S)\, dt \right| \le \epsilon |A_j| \quad\quad (n\ge n_0) .
\]
Two applications of \eqref{A.8} let us get rid of $y$ here. We obtain that
\[
\left| \int_{A_j} \omega_{P_-(n,t)(-\overline{m_j})}(-S)\, dt -
\int_{A_j} \omega_{m_-(n,t)}(-S)\, dt \right| \le 3\epsilon |A_j| \quad\quad (n\ge n_0).
\]
We combine this with \eqref{A.11}, \eqref{A.18}, \eqref{A.6}, then sum over $j=1,\ldots, N$
and finally recall
that $|A_0|<\epsilon$. It follows that
\[
\left| \int_A \omega_{m_+(n,t)}(S)\, dt -
\int_A \omega_{m_-(n,t)}(-S)\, dt \right| \le 4\epsilon |A| + 2 \epsilon
\]
if $n\ge n_0$.
\end{proof}
Let me try to summarize the argument. Consider $-\overline{m_+(n,t)}= -P_+(n,t)\overline{m_+(t)}$.
The crucial
identity \eqref{Abasic} shows that this latter expression equals $P_-(n,t)(-\overline{m_+(t)})$,
thus we also have that
\[
-\overline{m_+(n,t)} = P_-(n,t)(-\overline{m_+(t)}) .
\]
This already looks very similar to the reflectionless condition from Definition \ref{D1.1}, except
that $-\overline{m_+(t)}$ on the right-hand side
is not the correct initial value if we want to obtain $m_-(n,t)$ (that
would be $\infty$, as we saw in \eqref{A.3}). Fortunately, that doesn't
really matter, though, because of the focussing
property of the evolution of $m_-(n,z)$ that is expressed by Lemma \ref{LA.2} and Corollary \ref{CA.1}.
On second thoughts, things are not really that clear
because the evolution is focussing for $z\in\C^+$ and not for $z=t\in\R$.
However, Lemma \ref{LA.1} saves us then because it allows us to move into the upper half plane at low cost.

It is perhaps also illuminating to analyze why this proof doesn't prove too much (where do we need
that $A\subset\Sigma_{ac}$) and why the approximation of $m_+(t)$ by the step function with values
$m_j$ was necessary (it is actually not necessary if $m_+$ has a holomorphic continuation through
$A$ into the lower half plane). I will leave these points for the interested reader to explore.
\end{appendix}

\end{document}